\documentclass{birkjour}

\usepackage[utf8]{inputenc}
\usepackage[T1]{fontenc}
\usepackage[english]{babel}
\usepackage{mathtools,amsmath,amssymb,amsfonts,amsthm,mathrsfs}
\mathtoolsset{showonlyrefs}
\usepackage{bbm}
\usepackage{enumitem}

\usepackage{geometry}
\usepackage{graphicx,color}
\usepackage{tikz,pgfplots,pgf}
\usepackage[caption=false]{epsfig}
\usepackage{subcaption}
\usepackage[font=small]{caption}

\usepackage[bookmarksnumbered,colorlinks=true,urlcolor=black,linkcolor=black,citecolor=black,bookmarks, breaklinks]{hyperref}
\makeatletter
\newcommand{\leqnomode}{\tagsleft@true\let\veqno\@@leqno}
\makeatother

\newtheorem{theorem}{Theorem}[section]

\newtheorem{lemma}{Lemma}[section]
\newtheorem{proposition}{Proposition}[section]

\theoremstyle{definition}
\newtheorem{remark}{Remark}[section]

\numberwithin{equation}{section}

\begin{document}

\title[Homoclinic and heteroclinic solutions for Minkowski-curvature equations]{Homoclinic and heteroclinic solutions\\ for non-autonomous Minkowski-curvature equations}

\author[G.~Feltrin]{Guglielmo Feltrin}

\address{
Department of Mathematics, Computer Science and Physics, University of Udine\\
Via delle Scienze 206, 33100 Udine, Italy}
    
\email{guglielmo.feltrin@uniud.it}

\author[M.~Garrione]{Maurizio Garrione}

\address{
Department of Mathematics, Politecnico di Milano\\
Piazza Leonardo da Vinci, 32, 20133 Milano, Italy}

\email{maurizio.garrione@polimi.it}

\thanks{Work written under the auspices of the Grup\-po Na\-zio\-na\-le per l'Anali\-si Ma\-te\-ma\-ti\-ca, la Pro\-ba\-bi\-li\-t\`{a} e le lo\-ro Appli\-ca\-zio\-ni (GNAMPA) of the Isti\-tu\-to Na\-zio\-na\-le di Al\-ta Ma\-te\-ma\-ti\-ca (INdAM). The authors are supported by INdAM-GNAMPA project ``Analisi qualitativa di problemi differenziali non lineari''.
\\
\textbf{Preprint -- September 2023}} 

\subjclass{34C37, 35J93, 37J46.}

\keywords{Minkowski-curvature equation, heteroclinic solution, homoclinic solutions, asymptotic behaviour, phase-plane analysis}

\date{}

\dedicatory{Dedicated to Professor Fabio Zanolin on the occasion of his 70$^{\text{th}}$ birthday}

\begin{abstract}
We deal with the non-autonomous parameter-dependent second-order differential equation 
\begin{equation*}
\delta \biggl{(} \dfrac{v'}{\sqrt{1-(v')^{2}}} \biggr{)}' + q(t) f(v)= 0, \quad t\in\mathbb{R},
\end{equation*}
driven by a Minkowski-curvature operator. Here, $\delta>0$, $q\in L^{\infty}(\mathbb{R})$, $f\colon\mathopen{[}0,1\mathclose{]}\to\mathbb{R}$ is a continuous function with $f(0)=f(1)=0=f(\alpha)$ for some $\alpha \in \mathopen{]}0,1\mathclose{[}$, $f(s)<0$ for all $s\in\mathopen{]}0,\alpha\mathclose{[}$ and $f(s)>0$ for all $s\in\mathopen{]}\alpha,1\mathclose{[}$. 
Based on a careful phase-plane analysis, under suitable assumptions on $q$ we prove the existence of strictly increasing heteroclinic solutions and of homoclinic solutions with a unique change of monotonicity. Then, we analyze the asymptotic behaviour of such solutions both for $\delta \to 0^{+}$ and for $\delta\to+\infty$. Some numerical examples illustrate the stated results.
\end{abstract}

\maketitle

\section{Introduction}\label{section-1}

In this paper, we are concerned with homoclinic and heteroclinic solutions for the equation 
\begin{equation}\label{eq-intro}
\delta \biggl{(} \dfrac{v'}{\sqrt{1-(v')^{2}}} \biggr{)}' + q(t) f(v)= 0, \quad t\in\mathbb{R},
\end{equation}
where $\delta > 0$, $q \in L^{\infty}(\mathbb{R})$ and $f\colon \mathopen{[}0,1\mathclose{]} \to \mathbb{R}$ is a sign-changing function satisfying $f(0)=f(1)=0$.

The second-order operator appearing in \eqref{eq-intro}, given by $(\phi(v'))'$, with
\begin{equation}\label{op-intro}
\phi(\xi) = \dfrac{\xi}{\sqrt{1-\xi^{2}}},
\end{equation}
is usually found in the theory of nonlinear electromagnetism, where it is referred to as \emph{Born--Infeld operator}, and in general relativity, since it can be seen as a mean-curvature operator in the relativistic Lorentz--Minkowski space. We refer, for instance, to \cite{GaGuYa-19} and to the extensive discussions in \cite[p.~3]{BoCoNy-17} and in \cite{Az-16, Ga-22-pp} for further considerations in this respect. 

The investigation of homoclinic and heteroclinic solutions for second-order ODEs is a very classical topic; in the autonomous case we make reference, among the others, to \cite{BoSa-06,MaMa-02,MaMa-03} and to the bibliography in \cite{ArCa-22-pp}. In particular, a significant deal of attention has been received by such a problem in presence of nonlinear operators of curvature type, mainly as a byproduct of the search for travelling fronts, see, e.g.,  \cite{CoSa-14, GaSa-15, KuRo-06,Ru-20} and the references therein. Also in the non-autonomous case there are contributions, though in minor quantity; we mention, for instance, the papers~\cite{ElZa-13,PeWaLv-21} for equations governed by the linear second-order operator and \cite{BiCaPa-19,Ca-11} for more general problems dealt with through an abstract functional approach. 
In this respect, particularly significant in relation to the present manuscript is the paper \cite{BoCoNy-17}, where the authors make use of variational methods to find heteroclinics whenever $f$ is the derivative of a double-well potential.

The presence of the nonconstant weight $q$ in \eqref{eq-intro} makes indeed the considered problem non-autonomous and, in principle, prevents one from obtaining the desired solutions via a simple study of the orbits associated with the equivalent first-order system. Anyway, in this paper we will maintain a geometric phase-plane approach, aiming to obtain heteroclinics and homoclinics by gluing suitable branches of solutions. In this respect, it is useful to mention that, since $q$ is not necessarily continuous, by a solution of \eqref{eq-intro} we mean a continuously differentiable function $v\colon\mathbb{R}\to\mathopen{[}0,1\mathclose{]}$, with $v'$ absolutely continuous, which satisfies equation \eqref{eq-intro} almost everywhere. Moreover, if $v$ is strictly increasing with $v(-\infty)=0$, $v(+\infty)=1$, we say that $v$ is a \textit{heteroclinic solution}, while if $v(\pm\infty)=0$ and $v$ displays a unique change of monotonicity, we call $v$ a \textit{homoclinic solution}; in both cases $v'(\pm\infty)=0$ by the monotonicity. 
We also observe that in each interval of monotonicity of any solution $v=v(t)$, one can write the inverse function $t=t(v)$, so that $v(t(v))=v$, and regard $v$ as an independent variable. Setting
\begin{equation}\label{eq-first}
y(v) = \dfrac{1}{\sqrt{1-(v'(t(v)))^{2}}}-1
\quad
\left(\text{implying}
\quad
v'(t(v)) = \dfrac{\sqrt{(y(v))^{2}+2y(v)}}{y(v)+1}\right),
\end{equation}
we thus have
\begin{equation*}
\dfrac{\mathrm{d}}{\mathrm{d}v} y (v) = \dfrac{v''(t(v))}{(1-(v'(t(v)))^{2})^{\frac{3}{2}}} v'(t(v)) t'(v).
\end{equation*}
Since $v'(t(v)) t'(v)=1$, from \eqref{eq-intro} we conclude that $y$ satisfies 
\begin{equation}\label{eq-y}
\dot{y}(v) = - q(t(v))\frac{f(v)}{\delta},
\end{equation}
where from now on we denote by ``$\cdot$'' the differentiation with respect to $v$.
Noticing that $y(v)=0$ is equivalent to $v'(t(v))=0$ by \eqref{eq-first}, and $t(0)=-\infty$, $t(1)=+\infty$, the function $y$ defined by \eqref{eq-first} satisfies $y(0)=0$ and $y(1)=0$.
Again, the solution $y$ of \eqref{eq-y} is meant in the absolutely continuous sense, so that $\dot{y}$ is well defined almost everywhere.

As for the assumptions on $f$, when $q$ is constant the existence of heteroclinics and homoclinics necessarily requires that the primitive 
\begin{equation*}
F(v) := \int_{0}^{v} f(s)\,\mathrm{d}s
\end{equation*}
(always fulfilling $F(0)=0$) vanishes at some $v_{0} \in \mathopen{]}0, 1\mathclose{]}$, to which purpose $f$ has to change sign. To fix ideas, we consider a reaction term $f$ which is negative in a right neighborhood of $0$ and positive in a left neighborhood of $1$; more precisely, henceforth we assume that 
$f\colon \mathopen{[}0,1\mathclose{]} \to \mathbb{R}$ is a Lipschitz continuous function with Lipschitz constant $L > 0$, such that
\begin{enumerate}[leftmargin=26pt,labelsep=8pt,label=\textup{$(f_{1})$}]
\item $f(0)=f(1)=0$ and there exist $\alpha, \beta \in\mathopen{]}0,1\mathclose{[}$, with $\alpha \leq \beta$, such that $f(\alpha)=f(\beta)=0$, $f(s)<0$ for all $s\in\mathopen{]}0,\alpha\mathclose{[}$, $f(s)>0$ for all $s\in\mathopen{]}\beta,1\mathclose{[}$;
\label{hp-f-1}
\end{enumerate}
and
\begin{enumerate}[leftmargin=26pt,labelsep=8pt,label=\textup{$(f_{2})$}]
\item $F(1) > 0$;
\label{hp-f-2}
\end{enumerate}
or
\begin{enumerate}[leftmargin=26pt,labelsep=8pt,label=\textup{$(f_{2}')$}]
\item $F(1) =0$. 
\label{hp-f-2-b}
\end{enumerate}
Whenever useful, we will extend $f$ in a continuous way to the whole real line by setting $f(s)=0$ for all $s\in\mathbb{R} \setminus \mathopen{[}0,1\mathclose{]}$.
Notice that since $F(\alpha)<0$ and $F(1) \geq 0$, by the continuity of $F$ there exists $v_{0}\in\mathopen{]}\alpha,1\mathclose{]}$ such that $F(v_{0})=0$. 

Particularly common, in literature, is the case when $f$ is \emph{bistable}, that is, $\alpha=\beta$ and $f$ displays a single change of sign. Under this assumption, we can give a first result for a stepwise constant weight $q$ with a single jump, which can be immediately proved by elementary considerations in the phase-plane (see Figure~\ref{fig-01}). 

\begin{proposition}\label{prop-intro}
Let $\delta > 0$ be fixed and let $q\equiv c_{1}$ in $\mathopen{]}-\infty,t_{0}\mathclose{[}$ and $q\equiv c_{2}$ in $\mathopen{[}t_{0},+\infty\mathclose{[}$, with $c_{1},c_{2}>0$. Let $f$ be a Lipschitz continuous function satisfying \ref{hp-f-1} and \ref{hp-f-2}. 
Then, the following hold:
\begin{itemize}
\item if $c_{2} > c_{1}$, then any solution of \eqref{eq-intro} such that $v(-\infty)=0$ is ``definitively periodic'', that is, $v(t)=v(t+T)$ for every $t\in\mathopen{[}t_0, +\infty\mathclose{[}$ for a suitable $T>0$; 
\item if $c_{2} = c_{1}$, then there exists a homoclinic solution of \eqref{eq-intro}, unique up to $t$-translation;
\item if $c_{2} < c_{1}$ and $c_{2} \neq c_{1} \dfrac{-F(\rho)}{F(1)-F(\rho)}$ for all $\rho\in\mathopen{]}0,\alpha\mathclose{]}$, then all the solutions of \eqref{eq-intro} for which $v(-\infty)=0$ take either value $0$ or value $1$ (with nonzero derivative) in finite time;
\item if $c_{2} = c_{1} \dfrac{-F(\rho)}{F(1)-F(\rho)}$ for some $\rho\in\mathopen{]}0,\alpha\mathclose{]}$, then there exists a heteroclinic solution of \eqref{eq-intro}.
\end{itemize}
On the other hand, if \ref{hp-f-2-b} holds, then
\begin{itemize}
\item if $c_2 > c_1$, then any solution of \eqref{eq-intro} such that $v(-\infty)=0$ is ``definitively periodic'' in the above sense (in particular, there are no homoclinic solutions of \eqref{eq-intro});
\item if $c_{2} = c_{1}$, then there exists a heteroclinic solution of \eqref{eq-intro}; 
\item if $c_2 < c_1$, then all the solutions of \eqref{eq-intro} for which $v(-\infty)=0$ take either value $0$ or value $1$ (with nonzero derivative) in finite time.
\end{itemize}
\end{proposition}

\begin{figure}[htb]
\begin{tikzpicture}
\begin{axis}[
  scaled ticks=false,
  tick label style={font=\scriptsize},
  axis y line=left, axis x line=middle,
  xtick={0.4, 0.666667, 1},
  ytick={0},
  xticklabels={ , , $1$},
  yticklabels={$0$},
  xlabel={\small $v$},
  ylabel={\small $f(v)$},
every axis x label/.style={  at={(ticklabel* cs:1.0)},  anchor=west,
},
every axis y label/.style={  at={(ticklabel* cs:1.0)},  anchor=south,
},
  width=6cm,
  height=6cm,
  xmin=0,
  xmax=1.1,
  ymin=-0.05,
  ymax=0.08]
\addplot [color=black, fill=gray, fill opacity=0.2, line width=0.9pt,smooth] coordinates {(0., 0.) (0.02, -0.007448) (0.04, -0.013824) (0.06, -0.019176) (0.08, -0.023552) (0.1, -0.027) (0.12, -0.029568) (0.14, -0.031304) (0.16, -0.032256) (0.18, -0.032472) (0.2, -0.032) (0.22, -0.030888) (0.24, -0.029184) (0.26, -0.026936) (0.28, -0.024192) (0.3, -0.021) (0.32, -0.017408) (0.34, -0.013464) (0.36, -0.009216) (0.38, -0.004712) (0.4, 0.)};
\addplot [color=white, fill=gray, fill opacity=0.2, line width=0pt,smooth] coordinates {(0.4, 0.) (0.42, 0.004872) (0.44, 0.009856) (0.46, 0.014904) (0.48, 0.019968) (0.5, 0.025) (0.52, 0.029952) (0.54, 0.034776) (0.56, 0.039424) (0.58, 0.043848) (0.6, 0.048) (0.62, 0.051832) (0.64, 0.055296) (0.66, 0.058344) (0.666666667, 0.0592593) (0.666666667, 0)};
\addplot [color=black,line width=0.9pt,smooth] coordinates {(0., 0.) (0.02, -0.007448) (0.04, -0.013824) (0.06, -0.019176) (0.08, -0.023552) (0.1, -0.027) (0.12, -0.029568) (0.14, -0.031304) (0.16, -0.032256) (0.18, -0.032472) (0.2, -0.032) (0.22, -0.030888) (0.24, -0.029184) (0.26, -0.026936) (0.28, -0.024192) (0.3, -0.021) (0.32, -0.017408) (0.34, -0.013464) (0.36, -0.009216) (0.38, -0.004712) (0.4, 0.) (0.42, 0.004872) (0.44, 0.009856) (0.46, 0.014904) (0.48, 0.019968) (0.5, 0.025) (0.52, 0.029952) (0.54, 0.034776) (0.56, 0.039424) (0.58, 0.043848) (0.6, 0.048) (0.62, 0.051832) (0.64, 0.055296) (0.66, 0.058344) (0.68, 0.060928) (0.7, 0.063) (0.72, 0.064512) (0.74, 0.065416) (0.76, 0.065664) (0.78, 0.065208) (0.8, 0.064) (0.82, 0.061992) (0.84, 0.059136) (0.86, 0.055384) (0.88, 0.050688) (0.9, 0.045) (0.92, 0.038272) (0.94, 0.030456) (0.96, 0.021504) (0.98, 0.011368) (1., 0.)};
\draw [color=gray, dashed, line width=0.3pt] (axis cs: 0.66666667,0)--(axis cs: 0.66666667, 0.0592593);
\node at (axis cs: 0.42,-0.006) {\scriptsize{$\alpha$}};
\node at (axis cs: 0.67,-0.006) {\scriptsize{$v_{0}$}};
\end{axis}
\end{tikzpicture}
\quad\quad
\begin{tikzpicture}
\begin{axis}[
  scaled ticks=false,
  tick label style={font=\scriptsize},
  axis y line=left, axis x line=middle,
  xtick={0.4, 0.6666667},
  ytick={0},
  xticklabels={ ,  },
  yticklabels={$0$},
  xlabel={\small $v$},
  ylabel={\small $w$},
every axis x label/.style={  at={(ticklabel* cs:1.0)},  anchor=west,
},
every axis y label/.style={  at={(ticklabel* cs:1.0)},  anchor=south,
},
  width=6cm,
  height=6cm,
  xmin=0,
  xmax=1.1,
  ymin=-2,
  ymax=2]
\draw [color=gray, dashed, line width=0.3pt] (axis cs: 1,-2)--(axis cs: 1,2);
\addplot[thick,blue] graphics[xmin=0,ymin=-2,xmax=1,ymax=2] {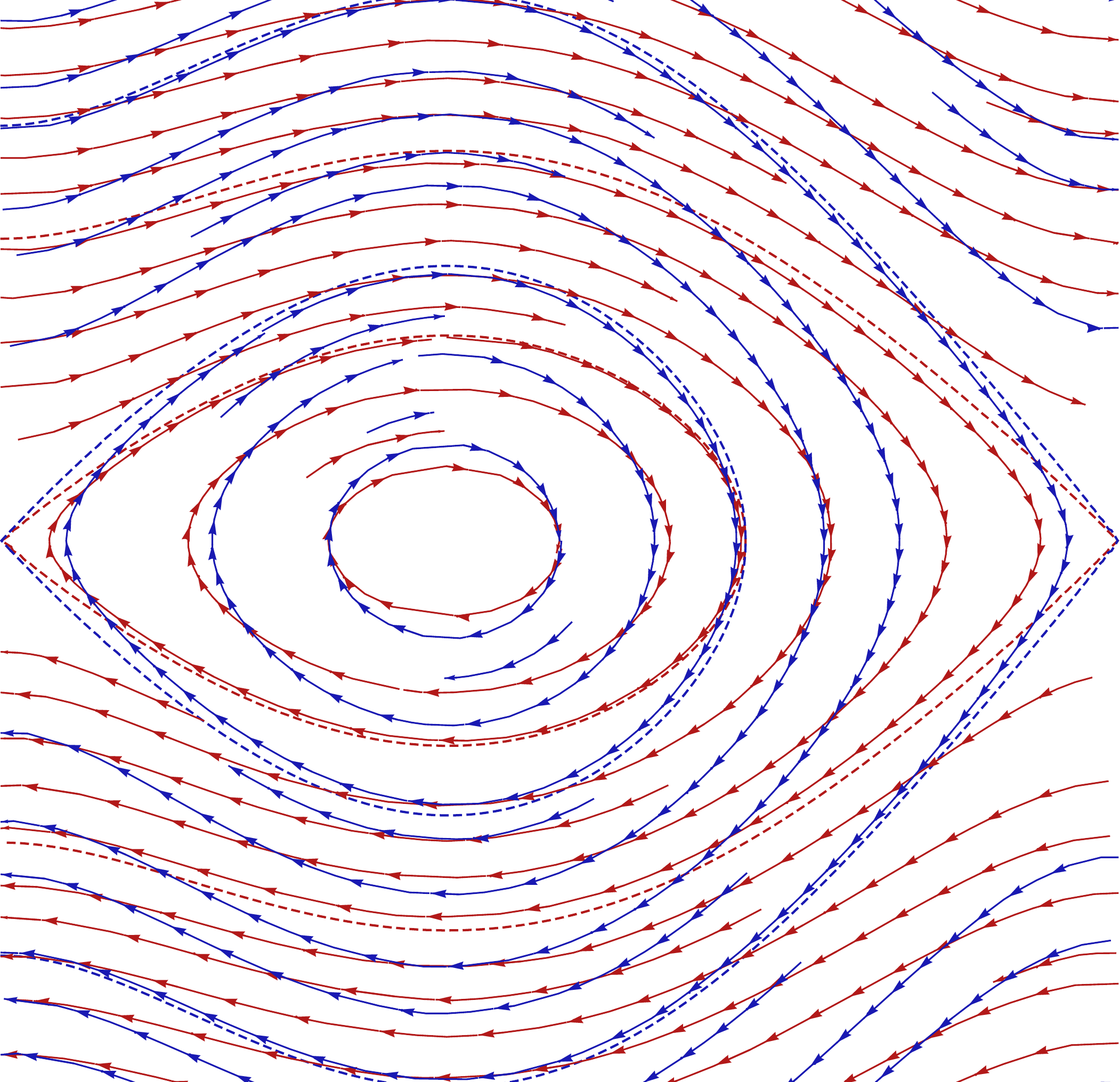};
\fill (axis cs: 0.4,0) circle (0.8pt);
\fill (axis cs: 0.6666667,0) circle (0.8pt);
\node at (axis cs: 0.4,-0.14) {\scriptsize{$\alpha$}};
\node at (axis cs: 0.7,-0.14) {\scriptsize{$v_{0}$}};
\node at (axis cs: 1.03,-0.14) {\scriptsize{$1$}};
\end{axis}
\end{tikzpicture}
\captionof{figure}{Qualitative graph of $f$ (on the left) and representation of the level lines $\mathcal{E}(v,w) = \sqrt{1+w^{2}} - 1 + c\dfrac{1}{\delta} F(v)$ for two values of the constant $c$ (on the right; the blue lines correspond to a higher value of $c$).}
\label{fig-01}
\end{figure}
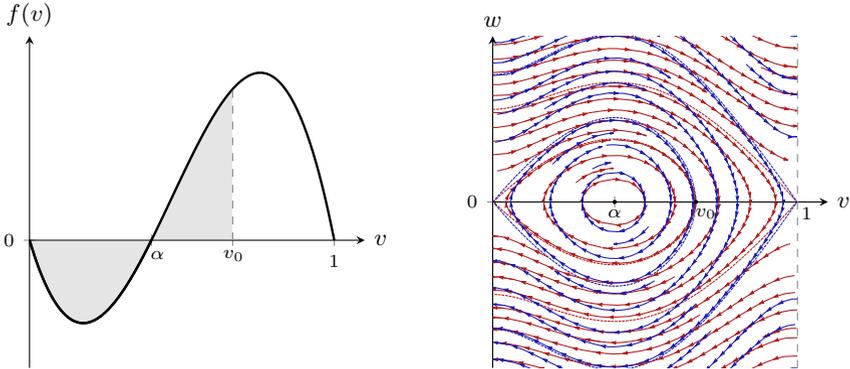

We explicitly remark that the statement of Proposition~\ref{prop-intro} does not depend on the fixed value of $\delta$. 

Our first goal is to extend Proposition~\ref{prop-intro} to the case of a nonconstant weight $q$ satisfying more general assumptions (see Theorems~\ref{th-main}--\ref{thm-nonesistenza}). In this respect, our results can be compared with the statements in \cite{BoCoNy-17}, where \ref{hp-f-2-b} is assumed (see Remark~\ref{rem-F(1)=0}).
The assumption of balancedness for $f$, exploited therein to reason through variational techniques, is however quite specific and does not survive under small perturbations of the reaction term, while we are here interested in results holding for general bistable nonlinearities. We thus seek heteroclinics and homoclinics adopting a different technique, based on a shooting method and on a precise phase-plane analysis, with the drawback of having to impose some more restrictive assumptions than in \cite{BoCoNy-17}. Since the problem is non-autonomous, we will indeed have to suitably control a family of branches of solutions in order to have or prevent intersections between them.

Subsequently, we investigate the asymptotic behavior of the constructed solutions for $\delta\to 0^{+}$ and for $\delta\to+\infty$; namely, by interpreting $\delta$ as a diffusion parameter, we consider a \emph{vanishing} or a \emph{large} diffusion limit, respectively (see, e.g., \cite{Ga-22-pp, HiKi-16} for a similar procedure in the framework of solutions of traveling front type).
The a priori bound $\vert v' \vert < 1$ for the derivative of regular solutions, coming from the expression of \eqref{eq-intro}, ensures that there is uniform convergence of the considered profiles to a Lipschitz continuous function, whose shape will here be our object of interest, on the lines of the considerations, e.g., in \cite{BoFe-20, Ga-22-pp, GaSo-23}. For the Minkowski operator, it is quite usual (even if some exceptions may arise, see for instance \cite{Ga-22-pp}) to expect limit profiles which are piecewise linear with slope $0$ or $1$, since the small parameter may compensate the diverging denominator of the second-order operator when $v' \to 1$. Indeed, we will prove a result of this kind (see Theorem~\ref{th-main-2} and Theorem~\ref{th-main-3}).

The main results of this paper are contained in Section~\ref{section-3}. The preceding Section~\ref{section-2}, where we review the autonomous case (i.e., $q\equiv1$), has the purpose of providing some motivation and some preliminary analysis for our study and exemplify, also through some pictures, the relative scenarios.

\section{Motivation: the autonomous case}\label{section-2}

In this section, we briefly review the existence and the qualitative properties of homoclinics and heteroclinics for \eqref{eq-intro} in the autonomous case $q \equiv 1$. The results are an immediate consequence of an elementary phase-plane analysis and will serve as a basis for the study in Section \ref{section-3}; for the reader's convenience, we will sometimes give some brief comments about the proofs, whenever not immediate.
Notice that here the solutions of \eqref{eq-intro} are of class $\mathcal{C}^{2}$.

To be more precise, recalling \eqref{op-intro}, we are dealing with the autonomous equation
\begin{equation}\label{eq-autonoma}
\delta \bigl{(} \phi(v'(t)) \bigr{)}' + f(v(t))= 0, \quad t\in\mathbb{R},
\end{equation}
or equivalently with the autonomous planar system
\begin{equation}\label{pl-syst}
\begin{cases}
\, v' = \phi^{-1}(w) = \dfrac{w}{\sqrt{1+w^{2}}},
\\
\, w' = -\dfrac{1}{\delta} f(v),
\end{cases}
\end{equation}
in dependence on the diffusion parameter $\delta > 0$. The associated energy function (vanishing at $(0,0)$) is given by
\begin{equation}\label{def-energy}
\mathcal{E}(v,w) = \sqrt{1+w^{2}} - 1 + \dfrac{1}{\delta} F(v).
\end{equation} 
Homoclinic (or heteroclinic, according to the assumption fulfilled by $F$) solutions for \eqref{eq-autonoma} are simply obtained by considering the orbit of \eqref{pl-syst} through $(0, 0)$, which intersects the $v$-axis in $v_{0}$ (recall that $v_{0}$ is such that $F(v_{0})=0$ and $v_{0}=1$ in case $F(1)=0$). Indeed, the Lipschitz continuity of $f$ guarantees that the equilibrium $(0, 0)$ can only be reached in infinite time. In other words, the time 
\begin{equation*}
T_{0,\delta} := \int_{0}^{v_{0}} \dfrac{\delta-F(v)}{\sqrt{F(v)(F(v)-2\delta)}} \,\mathrm{d}v
\end{equation*}
spent by the orbit to travel from $(0, 0)$ to $(v_{0}, 0)$ satisfies
\begin{equation*}
T_{0,\delta}=+\infty, \quad \text{for every $\delta > 0$.}
\end{equation*}
We first assume $F(1) > 0$, leaving to the end of the section some comments about the case of a \emph{balanced} reaction term (that is, $F(1)=0$). We preliminarily observe that, for fixed $\delta > 0$, any homoclinic to $0$ can naturally be seen as the limit of periodic solutions of \eqref{eq-autonoma}. Indeed, the assumptions on $f$ imply that, for $\gamma \in \mathopen{]}0,\alpha\mathclose{[}$, the function 
$F_\gamma(v):=F(v)-F(\gamma)$
has exactly two zeros in the interval $\mathopen{[}0,1\mathclose{]}$, given by $\gamma$ and a second value $\zeta(\gamma) \in \mathopen{]}\alpha, v_{0}\mathclose{[}$ (see also the left picture in Figure~\ref{fig-01}), for which it is clear that $\lim_{\gamma\to0^{+}} \zeta(\gamma)=v_{0}$.
For future convenience, we can extend the definition of $F_\gamma$ for $\gamma=0$ by setting $F_{0}:=F$.

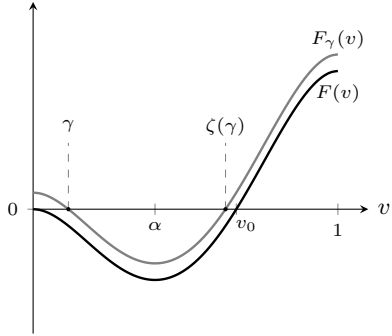
\begin{figure}[htb]
\begin{tikzpicture}
\begin{axis}[
  scaled ticks=false,
  tick label style={font=\scriptsize},
  axis y line=left, axis x line=middle,
  xtick={0.4, 0.6666667, 1},
  ytick={0},
  xticklabels={ , , $1$},
  yticklabels={$0$},
  xlabel={\small $v$},
  ylabel={},
every axis x label/.style={  at={(ticklabel* cs:1.0)},  anchor=west,
},
every axis y label/.style={  at={(ticklabel* cs:1.0)},  anchor=south,
},
  width=6cm,
  height=6cm,
  xmin=0,
  xmax=1.1,
  ymin=-0.015,
  ymax=0.025]
\draw [color=gray, dashed, line width=0.3pt] (axis cs: 0.115724,0)--(axis cs: 0.115724, 0.008);
\draw [color=gray, dashed, line width=0.3pt] (axis cs: 0.631374,0)--(axis cs: 0.631374, 0.008);
\addplot [color=black,line width=0.9pt,smooth] coordinates {(0., 0.) (0.02, -0.0000763067) (0.04, -0.000290773) (0.06,  -0.00062244) (0.08, -0.00105131) (0.1, -0.00155833) (0.12,  -0.00212544) (0.14, -0.00273551) (0.16, -0.00337237) (0.18,  -0.00402084) (0.2, -0.00466667) (0.22, -0.00529657) (0.24,  -0.00589824) (0.26, -0.00646031) (0.28, -0.00697237) (0.3,  -0.007425) (0.32, -0.00780971) (0.34, -0.00811897) (0.36,  -0.00834624) (0.38, -0.00848591) (0.4, -0.00853333) (0.42,  -0.00848484) (0.44, -0.00833771) (0.46, -0.00809017) (0.48,  -0.00774144) (0.5, -0.00729167) (0.52, -0.00674197) (0.54,  -0.00609444) (0.56, -0.00535211) (0.58, -0.00451897) (0.6,  -0.0036) (0.62, -0.00260111) (0.64, -0.00152917) (0.66,  -0.00039204) (0.68, 0.000801493) (0.7, 0.00204167) (0.72,  0.00331776) (0.74, 0.00461809) (0.76, 0.00593003) (0.78, 0.00723996) (0.8, 0.00853333) (0.82, 0.00979463) (0.84, 0.0110074) (0.86, 0.0121541) (0.88, 0.0132164) (0.9, 0.014175) (0.92, 0.0150095) (0.94, 0.0156986) (0.96, 0.0162202) (0.98, 0.0165509) (1., 0.0166667)};
\addplot [color=gray,line width=0.9pt,smooth] coordinates {(0., 0.002) (0.02, 0.00192369) (0.04, 0.00170923) (0.06, 0.00137756) (0.08, 0.000948693) (0.1, 0.000441667) (0.12, -0.00012544) (0.14, -0.000735507) (0.16, -0.00137237) (0.18, -0.00202084) (0.2, -0.00266667) (0.22, -0.00329657) (0.24, -0.00389824) (0.26, -0.00446031) (0.28, -0.00497237) (0.3, -0.005425) (0.32, -0.00580971) (0.34, -0.00611897) (0.36, -0.00634624) (0.38, -0.00648591) (0.4, -0.00653333) (0.42, -0.00648484) (0.44, -0.00633771) (0.46, -0.00609017) (0.48, -0.00574144) (0.5, -0.00529167) (0.52, -0.00474197) (0.54, -0.00409444) (0.56, -0.00335211) (0.58, -0.00251897) (0.6, -0.0016) (0.62, -0.000601107) (0.64, 0.000470827) (0.66, 0.00160796) (0.68, 0.00280149) (0.7, 0.00404167) (0.72, 0.00531776) (0.74, 0.00661809) (0.76, 0.00793003) (0.78, 0.00923996) (0.8, 0.0105333) (0.82, 0.0117946) (0.84, 0.0130074) (0.86, 0.0141541) (0.88, 0.0152164) (0.9, 0.016175) (0.92, 0.0170095) (0.94, 0.0176986) (0.96, 0.0182202) (0.98, 0.0185509) (1., 0.0186667)};
\node at (axis cs: 0.115724,0.01) {\scriptsize{$\gamma$}};
\node at (axis cs: 0.631374,0.01) {\scriptsize{$\zeta(\gamma)$}};
\fill (axis cs: 0.115724,0) circle (0.8pt);
\fill (axis cs: 0.631374,0) circle (0.8pt);
\node at (axis cs: 1,0.014) {\scriptsize{$F(v)$}};
\node at (axis cs: 1,0.0205) {\scriptsize{$F_{\gamma}(v)$}};
\node at (axis cs: 0.4,-0.002) {\scriptsize{$\alpha$}};
\node at (axis cs: 0.7,-0.002) {\scriptsize{$v_{0}$}};
\end{axis}
\end{tikzpicture}
\captionof{figure}{Qualitative graph of $F_{\gamma}$ for fixed $\gamma\in\mathopen{]}0,\alpha\mathclose{[}$, assuming $F(1) > 0$.}
\label{fig-02}
\end{figure}

The orbit of \eqref{pl-syst} passing through the points $(\gamma, 0)$ and $(\zeta(\gamma), 0)$ corresponds to a periodic solution of \eqref{eq-autonoma} having minimal period
\begin{equation}\label{def-Tgamma}
2T_{\gamma,\delta} = 2\int_{\gamma}^{\zeta(\gamma)} 
\dfrac{\delta-F_{\gamma}(v)}{\sqrt{F_{\gamma}(v)(F_{\gamma}(v)-2\delta)}} \,\mathrm{d}v.
\end{equation}
It is straightforward to check that $T_{\gamma, \delta}$ is finite for every $\gamma > 0$ and every $\delta > 0$; moreover, since
\begin{equation*}
T_{\gamma,\delta} = \int_{0}^{v_{0}} 
\dfrac{\delta-F_{\gamma}(v)}{\sqrt{F_{\gamma}(v)(F_{\gamma}(v)-2\delta)}} \, \mathbbm{1}_{\mathopen{[}\gamma,\zeta(\gamma)\mathclose{]}}(v)\,\mathrm{d}v,
\end{equation*}
(where $\mathbbm{1}_{\mathopen{[}\gamma,\zeta(\gamma)\mathclose{]}}$ denotes the indicator function of the interval $\mathopen{[}\gamma,\zeta(\gamma)\mathclose{]}$), a direct application of Fatou's lemma yields 
$\lim_{\gamma\to 0^{+}} T_{\gamma,\delta} = T_{0, \delta}= +\infty$. Therefore, denoting by $v_{\gamma, \delta}$ and $v_{0, \delta}$, respectively, the solutions of the problems
\begin{equation}\label{eq-v-gamma-delta}
\begin{cases}
\, \delta \bigl{(} \phi(v_{\gamma, \delta}') \bigr{)}' + f(v_{\gamma, \delta})= 0,
\\
\, v_{\gamma, \delta}(0)=\zeta(\gamma), \quad v_{\gamma, \delta}'(0)=0,
\end{cases}
\quad 
\begin{cases}
\, \delta \bigl{(} \phi(v_{0, \delta}') \bigr{)}' + f(v_{0, \delta})= 0,
\\
\, v_{0, \delta}(0)=v_{0}, \quad v_{0, \delta}'(0)=0,
\end{cases}
\end{equation}
the continuous dependence on the initial data ensures that 
$v_{\gamma,\delta} \to v_{0,\delta}$ in $\mathcal{C}^{2}_{\mathrm{loc}}(\mathbb{R})$ as $\gamma\to 0^{+}$, so that the homoclinic solution passing through $v_{0}$ can be seen as a (locally uniform) limit of $2T_{\gamma, \delta}$-periodic solutions. 

We now deepen our analysis of the asymptotic behavior of the periodic solutions $v_{\gamma, \delta}$ defined in \eqref{eq-v-gamma-delta} as $\delta$ and $\gamma$ vary.
First, we discuss the behavior of $v_{0, \delta}$ for $\delta \to 0^{+}$.

\begin{proposition}[$\gamma=0$ and $\delta\to 0^{+}$]\label{th-gamma-0-delta-0}
For $\delta\to 0^{+}$, it holds that $v_{0,\delta} \to \hat{v}_{0,0}$ locally uniformly in $t$, where
\begin{equation*}
\hat{v}_{0,0}(t) =
\begin{cases}
\, -t+v_{0}, &\text{if $t\in\mathopen{[}0,v_{0}\mathclose{]}$,}
\\
\, t+v_{0}, &\text{if $t\in\mathopen{[}-v_{0},0\mathclose{]}$,}
\\
\, 0, &\text{if $t\in\mathopen{]}-\infty,-v_{0}\mathclose{]}\cup\mathopen{[}v_{0},+\infty\mathclose{[}$.}
\end{cases}
\end{equation*}
\end{proposition}
For the proof, it turns useful to deal with \eqref{eq-y} via the change of variable \eqref{eq-first}, valid in each monotonicity interval of a solution $v$ of \eqref{eq-autonoma}. 
In particular, fixed $\gamma \in\mathopen{[}0,\alpha\mathclose{[}$ and $\delta>0$, we denote by $y_{\gamma,\delta}$ the unique solution of \eqref{eq-y} such that $y_{\gamma, \delta}(\gamma)=0$, namely
\begin{equation*}
y_{\gamma,\delta}(v) = -\dfrac{1}{\delta} F_{\gamma}(v), \quad v\in\mathopen{[}0,1\mathclose{]}.
\end{equation*}
We observe that $y_{\gamma,\delta}$ vanishes in correspondence of $\gamma$ and $\zeta(\gamma)$ and $\dot{y}_{\gamma,\delta}(v)=0$ if and only if $v\in\{0,\alpha,1\}$. 

\begin{proof}[Proof of Proposition~\ref{th-gamma-0-delta-0}]
We first notice that since $0 \leq v_{0, \delta}(t) \leq v_{0}$ and $\vert v_{0, \delta}'(t) \vert \leq 1$ for every $t \in \mathbb{R}$, by the Ascoli--Arzel\`a theorem we deduce that there exists a Lipschitz continuous function $\hat{v}_{0}$ such that $v_{0, \delta} \to \hat{v}_{0}$ locally uniformly in $\mathbb{R}$, for $\delta \to 0^{+}$. 
Moreover, $\hat{v}_{0}(0)=v_{0}>0$ and thus $\hat{v}_{0}\not\equiv0$. To simplify the argument, we now show that $\hat{v}_{0}(t)$ coincides with $\hat{v}_{0, 0}(t)$ for every $t \in \mathopen{]}-\infty, 0\mathclose{[}$, in order to take advantage of the positive sign of $v_{0, \delta}'$ therein; of course, a completely analogous argument works for $t \in \mathopen{]}0,+\infty\mathclose{[}$. 

By the discussion after formula \eqref{eq-y} and recalling that the homoclinic $v_{0, \delta}$ satisfies $v_{0, \delta}(0)=v_{0}$ and $v_{0, \delta}'(0)=0$, we then notice that the corresponding $y_\delta$ defined by \eqref{eq-first} satisfies the two-point problem 
\begin{equation*}
\begin{cases}
\, \dot{y}=-\dfrac{f(v)}{\delta},
\\
\, y(0)=0, \quad y(v_{0})=0,
\end{cases}
\end{equation*}
that is, $y_\delta(v)=-\frac{1}{\delta}F(v)$. Consequently, $y_\delta(v) \to +\infty$ as $\delta \to 0^{+}$ for every $v \in \mathopen{]}0, v_{0}\mathclose{[}$.
Since from \eqref{eq-y} one has that 
\begin{equation*}
v'_{0, \delta}(t)=\frac{\sqrt{y_\delta(v_{0, \delta}(t))(2+y_\delta(v_{0, \delta}(t)))}}{1+y_\delta(v_{0, \delta}(t))},
\end{equation*}
it follows that $\lim_{\delta \to 0^{+}} v'_{0, \delta}(t)=1$ for every $t \in \mathopen{]}-\infty, 0\mathclose{[}$ such that $\lim_{\delta \to 0^{+}} v_{0, \delta}(t) \not\in\{0, v_0\}$. 
Next we remark that the quantity
\begin{equation*}
\int_{\frac{v_{0}}{2}}^{v_{0}} \dfrac{\delta-F(v)}{\sqrt{F(v)(F(v)-2\delta)}} \,\mathrm{d}v,
\end{equation*}
representing the time needed for a solution $v$ to move from the value $v_{0}$ to the value $v_{0}/2$, is finite (and it converges for $\delta \to 0^{+}$, since it is monotone increasing with respect to $\delta$) and positive. Therefore, $v_{0, \delta}\not\equiv v_{0}$ in any neighbourhood of $0$ and, consequently, $|\hat{v}_{0}'|\equiv1$ in a neighbourhood of $0$.

Furthermore, $\{v_{0, \delta}'\}_\delta$ is bounded in $L^2_{\mathrm{loc}}(\mathbb{R})$, so (up to subsequences) it has a weak limit $w \in L^2_{\mathrm{loc}}(\mathbb{R})$ satisfying $0 \leq w \leq 1$, which coincides with the distributional derivative of $\hat{v}_{0}$. Thanks to the dominated convergence theorem, fixed an interval $\mathopen{[}t_{0}, t_1\mathclose{]} \subseteq \mathopen{]}-\infty,0\mathclose{[}$ we then have
\begin{align*}
\int_{t_{0}}^{t_1} \, \mathrm{d}s \geq \int_{t_{0}}^{t_1} w(s) \, \mathrm{d}s 
&= \hat{v}_{0}(t_1)-\hat{v}_{0}(t_{0}) =\lim_{\delta \to 0^{+}} (v_{0, \delta}(t_1)-v_{0, \delta}(t_{0}))
\\
&=\lim_{\delta \to 0^{+}}\int_{t_{0}}^{t_1} v_{0, \delta}'(s) \, \mathrm{d}s = \int_{t_{0}}^{t_1} \, \mathrm{d}s
 \end{align*}
and hence $w(t)=1$ for almost every $t \in [t_{0}, t_1]$. Being $\hat{v}_{0}$ absolutely continuous, for every $t\in\mathopen{]}-\infty,0\mathclose{[}$ we have that
\begin{equation*}
v_{0}-\hat{v}_{0}(t)=\hat{v}_{0}(0)-\hat{v}_{0}(t)=\int_t^{0} w(s) \,\mathrm{d}s=-t,
\end{equation*}
whence the conclusion, since $\hat{v}_{0}$ is non-decreasing in $\mathopen{]}-\infty,0\mathclose{]}$. The same argument holds for $t \in \mathopen{]}0,+\infty\mathclose{[}$, with reversed sign.
\end{proof}

We now discuss which picture appears inverting the way the two parameters $\gamma$ and $\delta$ converge to $0$: first, working at fixed $\gamma$ and sending $\delta \to 0^{+}$, we obtain the following.

\begin{proposition}[$\gamma\in\mathopen{]}0,\alpha\mathclose{[}$ and $\delta\to 0^{+}$]\label{th-delta-0}
For every $\gamma\in\mathopen{]}0,\alpha\mathclose{[}$, it holds that $v_{\gamma,\delta} \to v_{\gamma,0}$ locally uniformly in $t$ as $\delta\to 0^{+}$, where
\begin{equation}\label{def-vgamma0}
v_{\gamma,0}(t) =
\begin{cases}
\, -t+\zeta(\gamma), &\text{if $t\in\mathopen{[}0,\zeta(\gamma)-\gamma\mathclose{]}$,}
\\
\, t+\zeta(\gamma), &\text{if $t\in\mathopen{[}-\zeta(\gamma)+\gamma,0\mathclose{]}$,}
\end{cases}
\end{equation}
extended by $2(\zeta(\gamma)-\gamma)$-periodicity.
\end{proposition}

\begin{proof}
We can use an argument similar to the one in the previous proof to construct the limit profile $v_{\gamma, 0}$ of $v_{\gamma, \delta}$ for $\delta \to 0^{+}$ and to infer that its slope is everywhere equal to $\pm1$.
By the dominated convergence theorem, moreover, we can pass to the limit for $\delta\to0^{+}$ in \eqref{def-Tgamma} to find
\begin{equation*}
\tau:=\lim_{\delta\to 0^{+}} T_{\gamma,\delta} = \int_{\gamma}^{\zeta(\gamma)} 
\dfrac{-F_{\gamma}(v)}{\sqrt{(-F_{\gamma}(v))^{2}}} \,\mathrm{d}v = \zeta(\gamma)-\gamma.
\end{equation*}
On the other hand, passing to the limit for $\delta \to 0^{+}$ in the equality
$
 v_{\gamma, \delta}(0)-v_{\gamma, \delta}(-T_{\gamma, \delta}) = \zeta(\gamma)-\gamma,
$
using the uniform Lipschitz continuity of $v_{\gamma, \delta}$, we deduce
\begin{equation*}
v_{\gamma, 0}(0)-v_{\gamma, 0}(-\tau)=\zeta(\gamma)-\gamma. 
\end{equation*}
More in general, with the same argument one has, for every integer $k$, 
\begin{equation*}
v_{\gamma, 0}(2k\tau)-v_{\gamma, 0}((2k-1)\tau) = \zeta(\gamma)-\gamma, 
\quad 
v_{\gamma, 0}((2k+1)\tau)-v_{\gamma, 0}(2k\tau) = \gamma - \zeta(\gamma).
\end{equation*}
Being $\vert v_{\gamma, 0}'\vert \equiv 1$, $v_{\gamma,0}$ is periodic with minimal period $2(\zeta(\gamma)-\gamma)$ and the thesis follows.
\end{proof}

Finally, we consider the limit of $v_{\gamma, 0}$ for $\gamma \to 0^{+}$. 
\begin{proposition}[$\delta\to 0^{+}$ and $\gamma\to 0^{+}$]\label{th-delta-0-gamma-0}
It holds that $v_{\gamma,0} \to \check{v}_{0,0}$ uniformly in $t$, as $\gamma\to 0^{+}$, where
\begin{equation}\label{def-v00hat}
\check{v}_{0,0}(t) =
\begin{cases}
\, -t+v_{0}, &\text{if $t\in\mathopen{[}0,v_{0}\mathclose{]}$,}
\\
\, t+v_{0}, &\text{if $t\in\mathopen{[}-v_{0},0\mathclose{]}$,}
\end{cases}
\end{equation}
extended by $2v_{0}$-periodicity.
\end{proposition}

\begin{proof}
Since $\lim_{\gamma \to 0^{+}} \zeta(\gamma)=v_{0}$, we deduce $\lim_{\gamma \to 0^{+}} (\lim_{\delta \to 0^{+}} T_{\gamma,\delta})=\lim_{\gamma \to 0^{+}} (\zeta(\gamma)-\gamma)=v_{0}$ and the thesis follows.
\end{proof}

We have thus seen, on the one hand, that the profiles obtained in the limit for $\delta \to 0^{+}$ are piecewise linear, in accord with several results in literature for the Minkowski operator (like, e.g., \cite{BoFe-20}). On the other hand, exchanging the order with which the parameters $\gamma$ and $\delta$ are considered when computing the limit leads to different results. The above presented results are illustrated in Figure~\ref{fig-03}.

\begin{figure}[htb]
\centering
\begin{subfigure}[t]{0.45\textwidth}
\centering
\begin{tikzpicture}[scale=1]
\begin{axis}[
  scaled ticks=false,
  tick label style={font=\scriptsize},
  axis y line=middle, axis x line=bottom,
  xtick={-4,-3,-2,-1,0,1,2,3,4},
  ytick={0},
  xticklabels={ , , , ,$0$, $1$, , ,},
  yticklabels={$0$},
  xlabel={\small $t$},
  ylabel={\small $v(t)$},
every axis x label/.style={  at={(ticklabel* cs:1.0)},  anchor=west,
},
every axis y label/.style={  at={(ticklabel* cs:1.0)},  anchor=south,
},
  width=7cm,
  height=4cm,
  xmin=-5,
  xmax=5.2,
  ymin=0,
  ymax=1.1]
\addplot[thick,blue] graphics[xmin=-5,xmax=5,ymin=0,ymax=0.7] {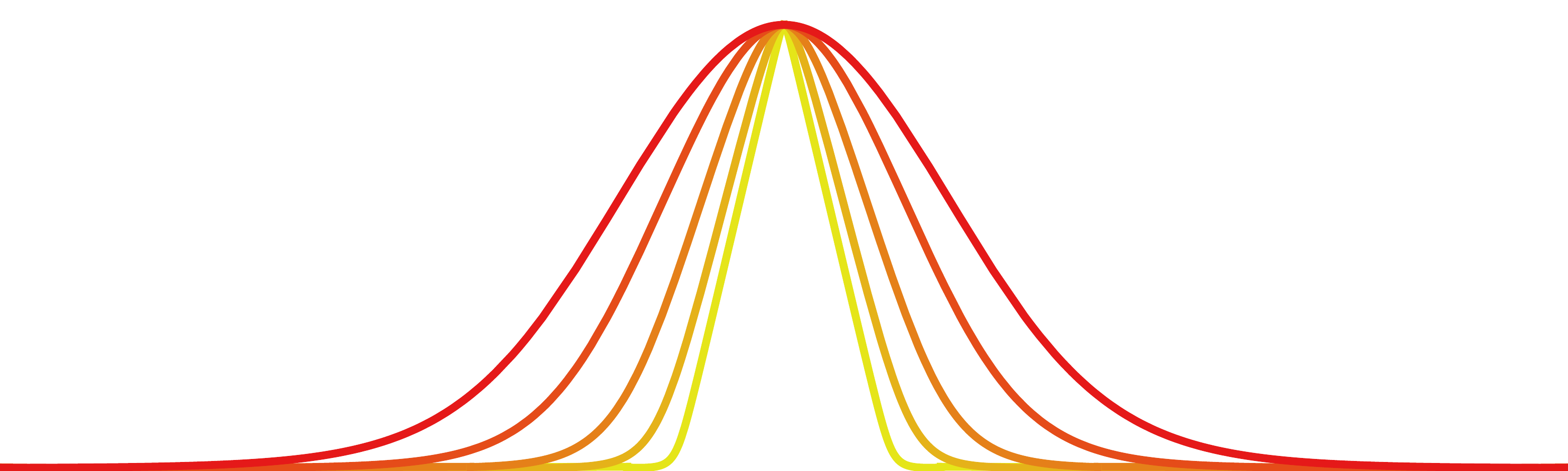};
\draw [color=gray, dashed, line width=0.3pt] (axis cs: -5,1)--(axis cs: 5,1);
\node at (axis cs: 0.4,0.73) {\scriptsize{$v_{0}$}};
\node at (axis cs: -0.25,1) {\scriptsize{$1$}};
\end{axis}
\end{tikzpicture}
\caption{Graphs of the homoclinic solutions $v_{0,\delta}$ of \eqref{eq-v-gamma-delta} ($\gamma=0$) for $\delta\in\{0.1,0.05,$\break$0.002,0.007,0.001\}$ (coloured from red to yellow).} 
\end{subfigure}
\hspace*{\fill}
\begin{subfigure}[t]{0.45\textwidth}
\centering
\begin{tikzpicture}[scale=1]
\begin{axis}[
  scaled ticks=false,
  tick label style={font=\scriptsize},
  axis y line=middle, axis x line=bottom,
  xtick={-10,-5,0,5,10},
  ytick={0},
  xticklabels={ , , $0$, $5$, },
  yticklabels={$0$},
  xlabel={\small $t$},
  ylabel={\small $v(t)$},
every axis x label/.style={  at={(ticklabel* cs:1.0)},  anchor=west,
},
every axis y label/.style={  at={(ticklabel* cs:1.0)},  anchor=south,
},
  width=7cm,
  height=4cm,
  xmin=-15,
  xmax=15.5,
  ymin=0,
  ymax=1.1]
\addplot[thick,blue] graphics[xmin=-15,xmax=15,ymin=0,ymax=0.7] {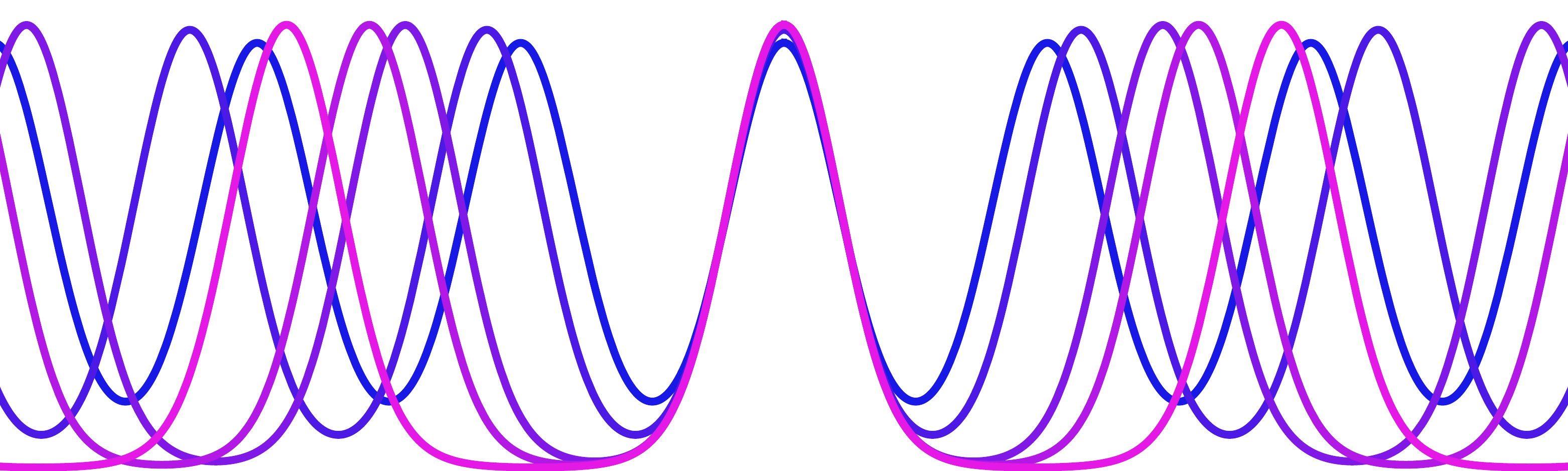};
\draw [color=gray, dashed, line width=0.3pt] (axis cs: -15,1)--(axis cs: 15,1);
\node at (axis cs: 1.3,0.73) {\scriptsize{$v_{0}$}};
\node at (axis cs: -0.25,1) {\scriptsize{$1$}};
\end{axis}
\end{tikzpicture}
\caption{Graphs of the homoclinic solutions $v_{\gamma,\delta}$ of \eqref{eq-v-gamma-delta} for $\delta=0.1$ and $\gamma\in\{0.1,$\break$0.05,0.01,0.005,0.001\}$ (coloured from purple to magenta).}
\end{subfigure}
\\
\vspace{15pt}
\begin{subfigure}[t]{0.45\textwidth}
\centering
\begin{tikzpicture}[scale=1]
\begin{axis}[
  scaled ticks=false,
  tick label style={font=\scriptsize},
  axis y line=middle, axis x line=bottom,
  xtick={-4,-3,-2,-1,0,1,2,3,4},
  ytick={0},
  xticklabels={ , , , ,$0$, $1$, , ,},
  yticklabels={$0$},
  xlabel={\small $t$},
  ylabel={\small $v(t)$},
every axis x label/.style={  at={(ticklabel* cs:1.0)},  anchor=west,
},
every axis y label/.style={  at={(ticklabel* cs:1.0)},  anchor=south,
},
  width=7cm,
  height=4cm,
  xmin=-5,
  xmax=5.2,
  ymin=0,
  ymax=1.1]
\addplot[thick,blue] graphics[xmin=-5,xmax=5,ymin=0,ymax=0.7] {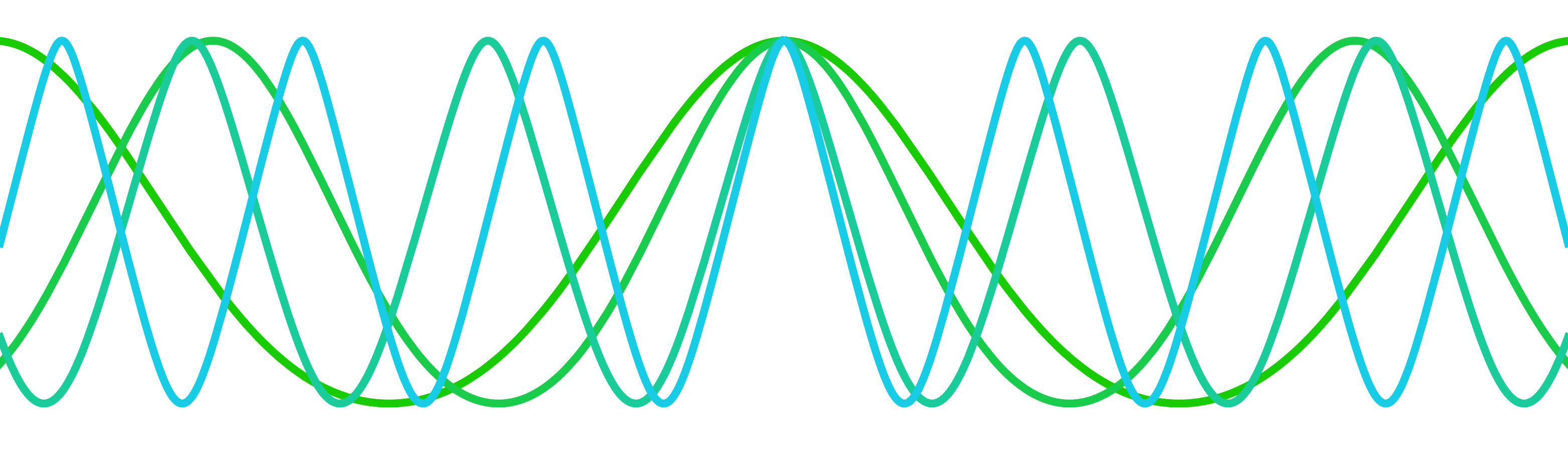};
\draw [color=gray, dashed, line width=0.3pt] (axis cs: -5,1)--(axis cs: 5,1);
\node at (axis cs: 0.6,0.73) {\scriptsize{$\zeta(\gamma)$}};
\node at (axis cs: -0.25,1) {\scriptsize{$1$}};
\end{axis}
\end{tikzpicture}
\caption{Graphs of the homoclinic solutions $v_{\gamma,\delta}$ of \eqref{eq-v-gamma-delta} for $\gamma=0.1$ and $\delta\in\{0.1,$\break$0.05,0.01,0.005\}$ (coloured from green to turquoise).}
\end{subfigure}
\hspace*{\fill}
\begin{subfigure}[t]{0.45\textwidth}
\centering
\begin{tikzpicture}[scale=1]
\begin{axis}[
  scaled ticks=false,
  tick label style={font=\scriptsize},
  axis y line=middle, axis x line=bottom,
  xtick={-4,-3,-2,-1,0,1,2,3,4},
  ytick={0},
  xticklabels={ , , , ,$0$, $1$, , ,},
  yticklabels={$0$},
  xlabel={\small $t$},
  ylabel={\small $v(t)$},
every axis x label/.style={  at={(ticklabel* cs:1.0)},  anchor=west,
},
every axis y label/.style={  at={(ticklabel* cs:1.0)},  anchor=south,
},
  width=7cm,
  height=4cm,
  xmin=-5,
  xmax=5.2,
  ymin=0,
  ymax=1.1]
\addplot[thick,blue] graphics[xmin=-5,xmax=5,ymin=0,ymax=0.7] {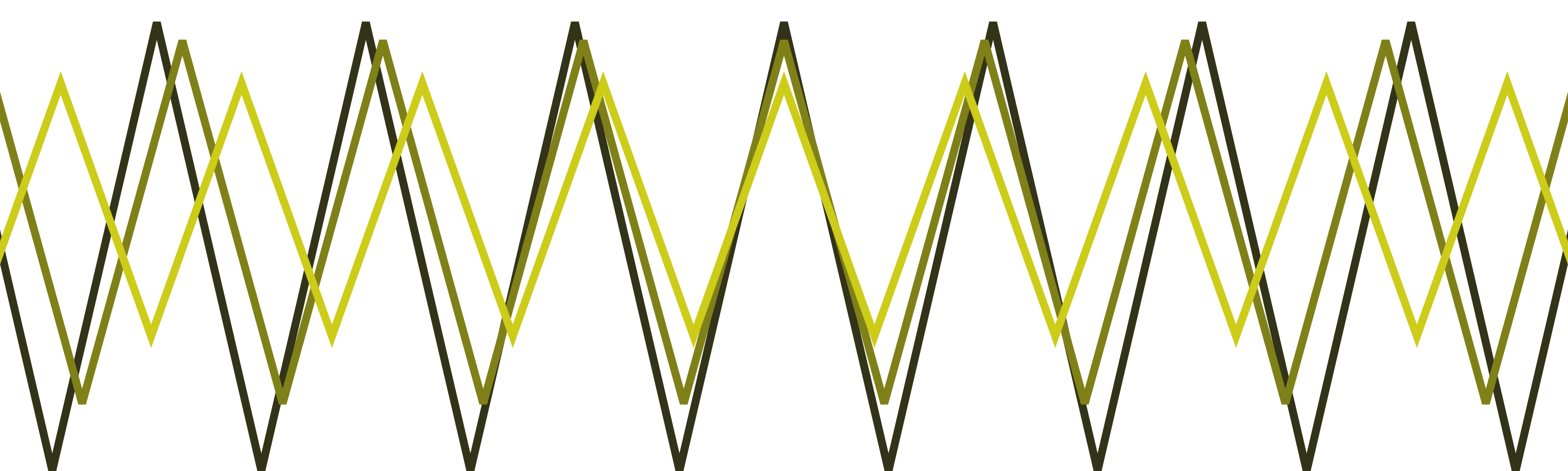};
\draw [color=gray, dashed, line width=0.3pt] (axis cs: -5,1)--(axis cs: 5,1);
\node at (axis cs: 0.45,0.73) {\scriptsize{$v_{0}$}};
\node at (axis cs: -0.25,1) {\scriptsize{$1$}};
\end{axis}
\end{tikzpicture} 
\caption{Qualitative graphs of the homoclinic solutions $v_{\gamma,\delta}$ of \eqref{eq-v-gamma-delta} for ``$\delta=0$'' and $\gamma\in\{0.2,0.1,0\}$ (coloured from ocher to dark green).}
\end{subfigure} 
\captionof{figure}{Qualitative representation of the the homoclinic solutions of \eqref{eq-autonoma} as $\delta$ and $\gamma$ tend to zero, for $f(s)=s(1-s)(s-0.4)$ (and so $F(1)>0$).}
\label{fig-03}
\end{figure}

\begin{remark}[The case $F(1)=0$]\label{rem-2.1}
In case $F(1)=0$ (and hence $v_{0}=1$), for $\gamma=0$ one would have, for every $\delta$, a heteroclinic connection between $0$ and $1$. In this case, with reference to the previous results, one cannot reason on the solution $v_{0, \delta}$ which satisfies $v_{0, \delta}(0)=1$, $v_{0, \delta}'(0)=0$, because by uniqueness this coincides with the constant function $1$. However, it is possible to proceed similarly as for the previous results by defining $v_{\gamma, \delta}$ as the solution of 
\begin{equation}\label{eq-heter-gamma-delta}
\begin{cases}
\, \delta \bigl{(} \phi(v_{\gamma, \delta}'(t)) \bigr{)}' + f(v_{\gamma, \delta}(t))= 0,
\\
\, v_{\gamma, \delta}(0)=\alpha, \quad v_{\gamma, \delta}'(0)=d_\alpha,
\end{cases}
\end{equation}
where $d_\alpha$ is such $\mathcal{E}(\alpha, d_\alpha)= \frac{1}{\delta}F(\gamma)$. 
In this way, $v_{\gamma, 0}$ will be the right shift of the function defined in \eqref{def-vgamma0} by the quantity $\zeta(\gamma) - \alpha$, whose limit for $\gamma \to 0^{+}$ coincides with the right shift of the function defined in \eqref{def-v00hat} by $1-\alpha$. On the other hand, $v_{0, \delta}$ will be the (increasing) heteroclinic connection between $0$ and $1$ such that $v_{0, \delta}(0)=\alpha$, which will then be approximated by means of the periodic solutions $v_{\gamma, \delta}$, having larger period the more $\gamma$ approaches $0$.  Finally, with the same proof as for Theorem~\ref{th-gamma-0-delta-0}, taking into account that $v_{0, \delta}$ is now everywhere increasing, one can show that $v_{0, \delta}$ converges locally uniformly to 
\begin{equation*}
\hat{v}_{0,0}(t) =
\begin{cases}
\, t+\alpha, &\text{if $t\in\mathopen{[}-\alpha,1-\alpha\mathclose{]}$,}
\\
\, 0, &\text{if $t\in\mathopen{]}-\infty,-\alpha\mathclose{]}$,}
\\
\, 1, &\text{if $t\in\mathopen{]}1-\alpha,+\infty\mathclose{]}$.}
\end{cases}
\end{equation*}
for $\delta \to 0^{+}$.
In Figure~\ref{fig-04}, we give a visual snapshot of these two convergences; the remaining two cases are similar to the ones depicted in Figure~\ref{fig-03} - $(\textsc{c})$ and Figure~\ref{fig-03} - $(\textsc{d})$, noticing that $\zeta(\gamma)\to1$ as $\gamma\to0^{+}$.
\hfill$\lhd$
\end{remark}

\begin{figure}[htb]
\centering
\begin{subfigure}[t]{0.45\textwidth}
\centering
\begin{tikzpicture}[scale=1]
\begin{axis}[
  scaled ticks=false,
  tick label style={font=\scriptsize},
  axis y line=middle, axis x line=bottom,
  xtick={-4,-3,-2,-1,0,1,2,3,4},
  ytick={0},
  xticklabels={ , , , ,$0$, $1$, , ,},
  yticklabels={$0$},
  xlabel={\small $t$},
  ylabel={\small $v(t)$},
every axis x label/.style={  at={(ticklabel* cs:1.0)},  anchor=west,
},
every axis y label/.style={  at={(ticklabel* cs:1.0)},  anchor=south,
},
  width=7cm,
  height=4cm,
  xmin=-5,
  xmax=5.2,
  ymin=0,
  ymax=1.1]
\addplot[thick,blue] graphics[xmin=-5,xmax=5,ymin=0,ymax=1] {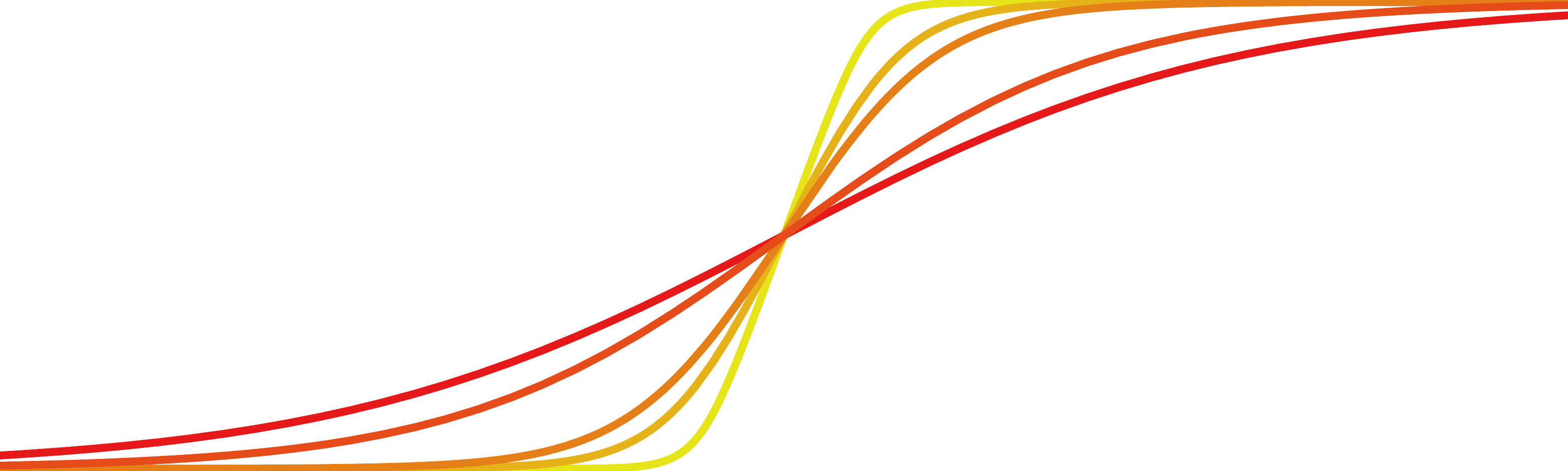};
\draw [color=gray, dashed, line width=0.3pt] (axis cs: -5,1)--(axis cs: 5,1);
\node at (axis cs: 0.4,0.46) {\scriptsize{$\alpha$}};
\node at (axis cs: -0.25,1) {\scriptsize{$1$}};
\end{axis}
\end{tikzpicture}
\caption{Graphs of the heteroclinic solutions $v_{0,\delta}$ of \eqref{eq-heter-gamma-delta} for $\delta\in\{1,0.5,0.1,$\break$0.05,0.01\}$ (coloured from red to yellow).} 
\end{subfigure}
\hspace*{\fill}
\begin{subfigure}[t]{0.45\textwidth}
\centering
\begin{tikzpicture}[scale=1]
\begin{axis}[
  scaled ticks=false,
  tick label style={font=\scriptsize},
  axis y line=middle, axis x line=bottom,
  xtick={-10,-5,0,5,10},
  ytick={0},
  xticklabels={ , , $0$, $5$, },
  yticklabels={$0$},
  xlabel={\small $t$},
  ylabel={\small $v(t)$},
every axis x label/.style={  at={(ticklabel* cs:1.0)},  anchor=west,
},
every axis y label/.style={  at={(ticklabel* cs:1.0)},  anchor=south,
},
  width=7cm,
  height=4cm,
  xmin=-15,
  xmax=15.5,
  ymin=0,
  ymax=1.1]
\addplot[thick,blue] graphics[xmin=-15,xmax=15,ymin=0,ymax=1] {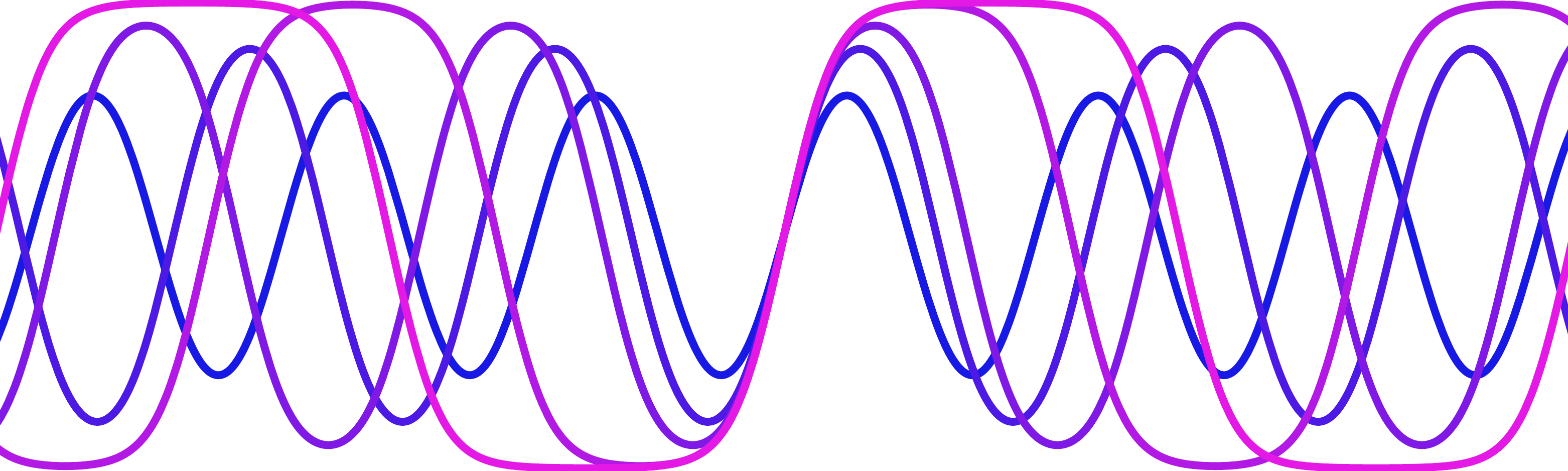};
\draw [color=gray, dashed, line width=0.3pt] (axis cs: -15,1)--(axis cs: 15,1);
\node at (axis cs: 1.1,0.46) {\scriptsize{$\alpha$}};
\node at (axis cs: -0.25,1) {\scriptsize{$1$}};
\end{axis}
\end{tikzpicture}
\caption{Graphs of the homoclinic solutions $v_{\gamma,\delta}$ of \eqref{eq-heter-gamma-delta} for $\delta=0.1$ and $\gamma\in\{0.2,0.1,$\break$0.05,0.005,0.0005\}$ (coloured from purple to magenta).}
\end{subfigure}
\captionof{figure}{Qualitative representation of the limit of the heteroclinic solutions of \eqref{eq-autonoma} as $\delta$ and $\gamma$ tend to zero, for $f(s)=s(1-s)(s-0.5)$ (and so $F(1)=0$).}
\label{fig-04}
\end{figure}

\section{A parametric problem with a non-constant positive weight}\label{section-3}

In this section, we deal with the non-autonomous differential equation
\begin{equation}\label{eq-sect-3}
\delta \bigl{(} \phi(v'(t)) \bigr{)}' + q(t) f(v(t))= 0,
\end{equation}
defined in $\mathbb{R}$, where $\delta>0$ and $q\in L^{\infty}(\mathbb{R})$ is a non-constant weight. We look for nontrivial homoclinic and heteroclinic solutions of \eqref{eq-sect-3}.

In more detail, in Section~\ref{section-3.1} we first provide the existence of solutions of \eqref{eq-sect-3} satisfying mixed Dirichlet-Neumann conditions at the boundary of a bounded interval. Next, in Section~\ref{section-3.2}, we determine the behavior of these solutions when one of the endpoints of the interval (the one with the Neumann condition) goes to $\pm\infty$. In Section~\ref{section-3.3}, we will exploit such a construction to find existence and non-existence results for homoclinic and heteroclinic solutions. At last, Section~\ref{section-3.4} is devoted to the investigation of the asymptotic behavior of these solutions as $\delta\to0^{+}$ and $\delta\to+\infty$.

\subsection{Boundary value problems in bounded intervals}\label{section-3.1}

Let $t_{0}\in\mathbb{R}$ and $T>0$.
We deal with the mixed Dirichlet-Neumann boundary value problems
\begin{equation}\label{mixed-1}
\begin{cases}
\, \delta \bigl{(} \phi(v') \bigr{)}' + q(t) f(v)= 0,
\\
\, v'(t_{0}-T) = 0, \quad v(t_{0})=\rho,
\end{cases}
\end{equation}
and 
\begin{equation}\label{mixed-2}
\begin{cases}
\, \delta \bigl{(} \phi(v') \bigr{)}' + q(t) f(v)= 0,
\\
\, v(t_{0}) = \rho, \quad v'(t_{0}+T)=0,
\end{cases}
\end{equation}
where $\rho\in\mathopen{]}0,1\mathclose{[}$. Notice that both \eqref{mixed-1} and \eqref{mixed-2} have mixed boundary conditions of Dirichlet-Neumann type.

Preliminarily, we show that every solution of \eqref{eq-sect-3}, with $(v(t_{0}-T),v'(t_{0}-T))=(\omega,0)$ and $\omega>0$ sufficiently small, remains (positive and) small and with positive derivative in $\mathopen{[}t_{0}-T,t_{0}\mathclose{]}$. 

\begin{lemma}\label{lem-gronwall-1}
Let $\delta>0$, $t_{0}\in\mathbb{R}$ and $T>0$.
Moreover, let $q\in L^{\infty}(t_{0}-T,t_{0})$ be a positive weight
and let $f\colon\mathopen{[}0, 1\mathclose{]} \to \mathbb{R}$ be a Lipschitz continuous function fulfilling \ref{hp-f-1}.
Then, for every $\gamma\in \mathopen{]}0,\alpha\mathclose{[}$ there exists $\omega_{\gamma}\in\mathopen{]}0,\alpha\mathclose{[}$ such that for every $\omega\in\mathopen{]}0,\omega_{\gamma}\mathclose{[}$ and for every solution $v$ of \eqref{eq-sect-3} with $(v(t_{0}-T),v'(t_{0}-T))=(\omega,0)$ it holds that
\begin{equation*}
(v(t),v'(t))\in\mathopen{]}0,\gamma\mathclose{[}\times\mathopen{]}0,+\infty\mathclose{[}, 
\quad \text{for every $t\in\mathopen{]}t_{0}-T,t_{0}\mathclose{]}$.}
\end{equation*}
\end{lemma}

\begin{proof}
Let $\gamma\in \mathopen{]}0,\alpha\mathclose{[}$ be fixed; we show that the statement holds choosing $\omega_{\gamma}\in\mathopen{]}0,\alpha\mathclose{[}$ which satisfies 
\begin{equation}\label{def-omega-gamma}
\omega_{\gamma} < \gamma\, e^{-\frac{1}{\delta} \|q\|_{L^{\infty}(t_{0}-T,t_{0})} L T^{2}},
\end{equation}
where $L>0$ is the Lipschitz constant of $f$.
To this end, let $\omega\in\mathopen{]}0,\omega_{\gamma}\mathclose{[}$ and $v$ be a solution of \eqref{eq-sect-3} with $(v(t_{0}-T),v'(t_{0}-T))=(\omega,0)$.
First, we write equation \eqref{eq-sect-3} in the equivalent form
\begin{equation*}
v''(t) + \dfrac{1}{\delta} q(t) f(v(t)) \bigl{(} 1- (v'(t))^{2} \bigr{)}^{\!\frac{3}{2}} = 0
\end{equation*}
and then we integrate twice in $\mathopen{[}t_{0}-T,t\mathclose{]}$, thus obtaining
\begin{align*}
v(t) &= v(t_{0}-T) + v'(t_{0}-T) (t-t_{0}+T) 
\\
&\quad - \dfrac{1}{\delta} \int_{t_{0}-T}^{t} \int_{t_{0}-T}^{s}q(\xi) f(v(\xi)) \bigl{(} 1- (v'(\xi))^{2} \bigr{)}^{\!\frac{3}{2}}
\,\mathrm{d}\xi\,\mathrm{d}s
\\
&= \omega + 0 - \dfrac{1}{\delta} \int_{t_{0}-T}^{t}  q(\xi) f(v(\xi)) \bigl{(} 1- (v'(\xi))^{2} \bigr{)}^{\!\frac{3}{2}} (t-\xi)
\,\mathrm{d}\xi,
\end{align*}
for every $t\in\mathopen{[}t_{0}-T,t_{0}\mathclose{]}$, where the last equality follows by an application of Fubini's theorem.
Next, using the fact that $f$ is Lipschitz continuous and $\|v'\|_{\infty}<1$, we have that
\begin{equation*}
v(t) \leq \omega + \dfrac{1}{\delta} \|q\|_{L^{\infty}(t_{0}-T,t_{0})} L (t-t_{0}+T) \int_{t_{0}-T}^{t}  v(\xi)\,\mathrm{d}\xi,
\quad \text{for every $t\in\mathopen{[}t_{0}-T,t_{0}\mathclose{]}$.}
\end{equation*}
The Gr\"{o}nwall's inequality and \eqref{def-omega-gamma} imply that
\begin{equation*}
v(t) \leq \omega \, e^{\frac{1}{\delta}\|q\|_{L^{\infty}(t_{0}-T,t_{0})} L (t-t_{0}+T)^{2}} < \gamma, 
\quad \text{for every $t\in\mathopen{]}t_{0}-T,t_{0}\mathclose{]}$.}
\end{equation*}
Therefore, $f(v(t))<0$ for every $t\in\mathopen{]}t_{0}-T,t_{0}\mathclose{]}$ and thus
\begin{equation*}
v''(t) = -\dfrac{1}{\delta}q(t) f(v(t)) \bigl{(} 1- (v'(t))^{2} \bigr{)}^{\!\frac{3}{2}} >0, \quad \text{for almost every $t\in\mathopen{[}t_{0}-T,t_{0}\mathclose{]}$.}
\end{equation*}
As a consequence, we deduce
\begin{equation*}
v'(t) = v'(t_{0}-T)+\int_{t_{0}-T}^{t} v''(\xi)\,\mathrm{d}\xi >0, \quad \text{for every $t\in\mathopen{]}t_{0}-T,t_{0}\mathclose{]}$.}
\end{equation*}
The proof is complete.
\end{proof}

We can draw analogous considerations regarding the solutions ``starting near~$1$''. Precisely, with the sole change consisting in integrating on $[t_{0}, t_{0}+T]$ instead of $[t_{0}-T, t_{0}]$, it is possible to show that for every $\gamma \in \mathopen{]}\beta,1\mathclose{[}$ one can find $\omega_{\gamma} \in \mathopen{]}\beta, 1 \mathclose{[}$ such that if $\omega\in\mathopen{]}\omega_{\gamma}, 1\mathclose{[}$, then the solution $v$ of \eqref{eq-sect-3} satisfying $(v(t_{0}+T),v'(t_{0}+T))=(\omega,0)$ is positive and increasing for every $t \in [t_{0}, t_{0}+T]$, as a result of the fact that it is concave and arrives with zero derivative at the time instant $t_{0}+T$. It is then possible to state the following result.

\begin{lemma}\label{lem-gronwall-2}
Let $\delta > 0, t_{0}\in\mathbb{R}$ and $T>0$.
Moreover, let $q\in L^{\infty}(t_{0},t_{0}+T)$ be a positive weight and let $f \colon\mathopen{[}0, 1\mathclose{]} \to \mathbb{R}$ be a Lipschitz continuous function fulfilling \ref{hp-f-1}.
Then, for every $\gamma\in \mathopen{]}\beta,1\mathclose{[}$ there exists $\omega_{\gamma}\in\mathopen{]}\beta,1\mathclose{[}$ such that for every $\omega\in\mathopen{]}\omega_{\gamma},1\mathclose{[}$ and for every solution $v$ of \eqref{eq-sect-3} with $(v(t_{0}+T),v'(t_{0}+T))=(\omega,0)$ it holds that
\begin{equation*}
(v(t),v'(t))\in\mathopen{]}\gamma,1\mathclose{[}\times\mathopen{]}0, +\infty\mathclose{[}, 
\quad \text{for every $t\in\mathopen{[}t_{0},t_{0}+T\mathclose{[}$.}
\end{equation*}
\end{lemma}

Next, we prove the existence of a (positive) strictly increasing solution of the boundary value problem \eqref{mixed-1}.

\begin{theorem}\label{th-exist-mixed-1}
Let $\delta > 0$, $t_{0}\in\mathbb{R}$ and $T>0$. Moreover, let $q\in L^{\infty}(t_{0}-T,t_{0})$ be a positive weight and let $f\colon\mathopen{[}0, 1\mathclose{]} \to \mathbb{R}$ be a Lipschitz continuous function fulfilling \ref{hp-f-1}. 
Then, for every $\rho\in\mathopen{]}0,\alpha\mathclose{[}$, there exists a strictly increasing solution of problem \eqref{mixed-1}.
\end{theorem}

\begin{proof}
The proof is based on a \textit{shooting technique} in the phase-plane $(v,w)=(v,\phi(v'))$; we divide it into two steps.

\smallskip

\noindent
\textit{Step~1. Existence.}
First, as mentioned in the Introduction, we extend the function $f$ continuously to the whole real line by setting $f(v)=0$ for $v\in\mathopen{]}-\infty,0\mathclose{[}\cup  \mathopen{]}1,+\infty\mathclose{[}$, still denoting such an extension by $f$.
Accordingly, we consider the planar system
\begin{equation}\label{syst-sect-3}
\begin{cases}
\, v' = \phi^{-1}(w),
\\
\, w' = - \dfrac{1}{\delta} q(t) f(v),
\end{cases}
\end{equation}
which is equivalent to the differential equation in \eqref{mixed-1}.
Since the function $f$ is Lipschitz continuous, the solutions of the associated Cauchy problems are globally defined on any compact time interval. Thus, fixed $[t_1, t_2] \subseteq \mathbb{R}$ with $t_1 < t_2$, we can introduce the associated Poincar\'{e} map $\mathcal{P}_{t_{1}}^{t_{2}}\colon \mathbb{R}^{2} \to \mathbb{R}^{2}$, which is the global diffeomorphism of the plane onto itself defined by
\begin{equation*}
\mathcal{P}_{t_{1}}^{t_{2}}(v_{1},w_{1}) := 
(v(t_{2};t_{1},v_{1},w_{1}),w(t_{2};t_{1},v_{1},w_{1}));
\end{equation*}
here, $(v(\cdot;t_{1},v_{1},w_{1}),w(\cdot;t_{1},v_{1},w_{1}))$ is the unique solution of \eqref{syst-sect-3} satisfying the initial condition $(v(t_{1}),w(t_{1}))=(v_{1},w_{1})$.
Our goal is to describe the deformation of the set $\mathopen{[}0,1\mathclose{]}\times\{0\}$ in the phase-plane $(v,w)$, through the Poincar\'{e} map $\mathcal{P}_{t_{0}-T}^{t_{0}}$. 

To this end, we first recall that $v\equiv 0$ and $v\equiv\alpha$ are trivial solution of \eqref{syst-sect-3}, so that
\begin{equation*}
\mathcal{P}_{t_{0}-T}^{t_{0}}(0,0) = (0,0), \quad \mathcal{P}_{t_{0}-T}^{t_{0}}(\alpha,0) = (\alpha,0);
\end{equation*}
therefore, by a continuity argument, for every $\rho\in\mathopen{]}0,\alpha\mathclose{[}$ we deduce that there exists $\omega_{\rho}\in\mathopen{]}0,\alpha\mathclose{[}$ such that
\begin{equation*}
\mathcal{P}_{t_{0}-T}^{t_{0}}(\omega_{\rho},0) \in \{\rho\}\times \mathbb{R}
\end{equation*}
(see Figure~\ref{fig-05} for a qualitative representation of the phase-plane).
We conclude that the first component of $(v(\cdot;t_{0}-T,\omega_{\rho},0),w(\cdot;t_{0}-T,\omega_{\rho},0))$ is a solution of the boundary value problem \eqref{mixed-1}.

\begin{figure}[htb]
\begin{tikzpicture}
\begin{axis}[
  scaled ticks=false,
  tick label style={font=\scriptsize},
  axis y line=left, axis x line=middle,
  xtick={0.45, 0.7, 1},
  ytick={0},
  xticklabels={$\rho$, $\alpha$, $1$},
  yticklabels={$0$},
  xlabel={\small $v$},
  ylabel={\small $\phi(v')$},
every axis x label/.style={  at={(ticklabel* cs:1.0)},  anchor=west,
},
every axis y label/.style={  at={(ticklabel* cs:1.0)},  anchor=south,
},
  width=6cm,
  height=5cm,
  xmin=0,
  xmax=1.1,
  ymin=-0.5,
  ymax=1]
\draw [color=gray, dashed, line width=0.3pt] (axis cs: 0.45,0)--(axis cs: 0.45, 1);
\addplot [color=magenta,line width=0.9pt,smooth] coordinates {(0., 0.) (0.5,0.7) (0.7, 0) (0.9, -0.3) (1, 0)};
\end{axis}
\end{tikzpicture}
\captionof{figure}{Representation of the deformation of the set $\mathopen{[}0,1\mathclose{]}\times\{0\}$ in the phase-plane $(v,w)=(v,\phi(v'))$, through the Poincar\'{e} map $\mathcal{P}_{t_{0}-T}^{t_{0}}$.}
\label{fig-05}
\end{figure}
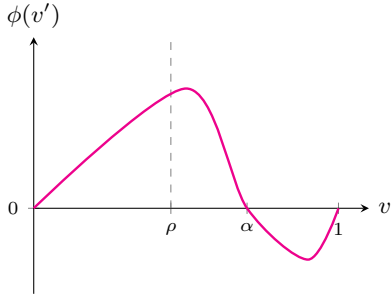

\smallskip

\noindent
\textit{Step~2. Monotonicity.}
Let $\rho\in\mathopen{]}0,\alpha\mathclose{[}$.
We aim to prove that the ``first intersection'' between the continuum $\mathcal{P}_{t_{0}-T}^{t_{0}}(\mathopen{[}0,\alpha\mathclose{]}\times\{0\})$ and $\{\rho\}\times \mathbb{R}$ corresponds to a strictly increasing solution of \eqref{mixed-1}.
Accordingly, let
\begin{equation*}
\hat{\omega}_{\rho} := \inf \bigl{\{} \omega\in\mathopen{]}0,\alpha\mathclose{[} \colon \mathcal{P}_{t_{0}-T}^{t_{0}}(\omega,0) \in \{\rho\}\times \mathbb{R} \bigr{\}}
\end{equation*}
and let $\hat{v}(\cdot)=\hat{v}(\cdot;\hat{\omega}_{\rho},0)$ be the solution of \eqref{eq-sect-3} with initial condition $(v(t_{0}-T),v'(t_{0}-T))=(\hat{\omega}_{\rho},0)$.
We claim that 
\begin{equation}\label{claim-monotonicity}
\hat{v}(t)\in\mathopen{[}0,\rho\mathclose{]}, \quad \text{for all $t\in\mathopen{]} t_{0}-T, t_{0} \mathclose{]}$.} 
\end{equation}
In order to prove it, let 
\begin{equation*}
\tilde{\omega}_{\rho} := \sup \bigl{\{} \omega\in\mathopen{]}0,\alpha\mathclose{[} \colon \mathcal{P}_{t_{0}-T}^{t}(\nu,0) \in\mathopen{[}0,\rho\mathclose{[}\times \mathbb{R}, \; \forall \, t\in \mathopen{[}t_{0}-T, t_{0} \mathclose{]}, \forall \, \nu\in \mathopen{[}0,\omega \mathclose{]}\bigr{\}}.
\end{equation*}
Notice that $\tilde{\omega}_{\rho}$ is well-defined, since, by an application of Lemma~\ref{lem-gronwall-1} (with $\gamma=\rho$), we deduce that there exists $\nu_{\rho}\in\mathopen{]}0,\alpha\mathclose{[}$ such that 
\begin{equation*}
\mathcal{P}_{t_{0}-T}^{t}(\mathopen{[}0,\nu_{\rho}\mathclose{]}\times\{0\}) \subseteq \mathopen{[}0,\rho\mathclose{[}\times \mathopen{]}0,+\infty\mathclose{[}, \quad \text{for all $t\in \mathopen{[}t_{0}-T, t_{0} \mathclose{]}.$}
\end{equation*}
If we prove that $\hat{\omega}_{\rho} = \tilde{\omega}_{\rho}$, then \eqref{claim-monotonicity} follows.
It is obvious that $\tilde{\omega}_{\rho} \leq \hat{\omega}_{\rho}$. Let us suppose, by contradiction, that $\tilde{\omega}_{\rho} < \hat{\omega}_{\rho}$. 
Then, since $\tilde{\omega}_{\rho}<\alpha$ (being $\mathcal{P}_{t_{0}-T}^{t}(\alpha,0)=(\alpha,0)$), due to the definition of $\tilde{\omega}_{\rho}$, for all $\nu \leq \tilde{\omega}_{\rho}$, the solution $v(\cdot; \nu, 0)$ is such that 
\begin{equation}\label{eq-stimaphi}
\phi(v'(t; \nu, 0)) = - \frac{1}{\delta} \int_{t_{0}-T}^{t} q(\xi) f(v(\xi; \nu, 0))\,\mathrm{d}\xi > 0, \quad \text{for every $t\in\mathopen{]} t_{0}-T, t_{0} \mathclose{]}$,}
\end{equation}
thanks to assumption \ref{hp-f-1}.
Hence, $v'(t; \nu, 0)>0$ for all $t\in \mathopen{[}t_{0}-T, t_{0} \mathclose{]}$, whence $\max_{\mathopen{[}t_{0}-T, t_{0} \mathclose{]}} v(t; \nu, 0) = v(t_{0}; \nu, 0)$. In particular,
\begin{equation}\label{eq-3-new}
v(t_{0}; \tilde{\omega}_{\rho}, 0)=\max_{t\in\mathopen{[}t_{0}-T, t_{0} \mathclose{]}} v(t; \tilde{\omega}_{\rho}, 0)= \rho;
\end{equation}
the last equality holds true since otherwise the continuous dependence with respect to the initial data would imply the existence of $\tilde{\omega}_{\rho}'>\tilde{\omega}_{\rho}$ such that $\mathcal{P}_{t_{0}-T}^{t}(\nu,0) \in\mathopen{[}0,\rho\mathclose{[}\times \mathbb{R}$, for every $t\in \mathopen{[}t_{0}-T, t_{0} \mathclose{]}$ and $\nu\in \mathopen{[}0,\tilde{\omega}_{\rho}' \mathclose{]}$ contradicting the definition of $\tilde{\omega}_{\rho}$. We observe that \eqref{eq-3-new} and the assumption $\tilde{\omega}_{\rho} < \hat{\omega}_{\rho}$ lead to a contradiction with the definition of $\hat{\omega}_{\rho}$.
 The claim \eqref{claim-monotonicity} is proved. From \eqref{claim-monotonicity} and \eqref{eq-stimaphi}, we then have that
$\hat{v}'(t)>0$ for every $t\in\mathopen{]} t_{0}-T, t_{0} \mathclose{]}$, implying that $\hat{v}$ is a strictly increasing solution of \eqref{mixed-1}. The thesis follows.
\end{proof}

In a similar manner, working with the Poincar\'e map $\mathcal{P}_{t_{0}+T}^{t_{0}}$, one can prove the following existence result for the boundary value problem \eqref{mixed-2}. We omit the proof, which is similar to the one for Theorem~\ref{th-exist-mixed-1}.

\begin{theorem}\label{th-exist-mixed-2}
Let $\delta > 0$, $t_{0}\in\mathbb{R}$ and $T>0$. Moreover, let $q\in L^{\infty}(t_{0},t_{0}+T)$ be a positive weight and let $f\colon\mathopen{[}0, 1\mathclose{]} \to \mathbb{R}$ be a Lipschitz continuous function fulfilling \ref{hp-f-1}. 
Then, for every $\rho\in\mathopen{]}\beta,1\mathclose{[}$, there exists a strictly increasing solution of problem \eqref{mixed-2}. 
\end{theorem}

\subsection{Solutions in unbounded intervals: passing to the limit for $T\to+\infty$}\label{section-3.2}

Let $t_{0}\in\mathbb{R}$ and consider the differential problems
\begin{equation}\label{bvp-1}
\begin{cases}
\, \delta \bigl{(} \phi(v') \bigr{)}' + q(t) f(v)= 0,
\\
\, v(-\infty) = 0, \quad v(t_{0})=\rho,
\end{cases}
\end{equation}
and 
\begin{equation}\label{bvp-2}
\begin{cases}
\, \delta \bigl{(} \phi(v') \bigr{)}' + q(t) f(v)= 0,
\\
\, v(t_{0}) = \rho, \quad v(+\infty)=1,
\end{cases}
\end{equation}
where $\rho\in\mathopen{]}0,1\mathclose{[}$. 
We prove that the limits of the solutions of \eqref{mixed-1} and of \eqref{mixed-2} for $T\to+\infty$ solve, respectively, \eqref{bvp-1} and \eqref{bvp-2}.

\begin{theorem}\label{th-exist-bvp-1}
Let $\delta > 0$ and $t_{0}\in\mathbb{R}$. Moreover, let $q\in L^{\infty}(-\infty,t_{0})$ be a positive weight such that $\|q\|_{L^{1}(t,t_{0})}\to+\infty$ as $t\to-\infty$ and assume that $f \colon \mathopen{[}0, 1\mathclose{]} \to \mathbb{R}$ is a Lipschitz continuous function fulfilling \ref{hp-f-1}. 
Then, for every $\rho\in\mathopen{]}0,\alpha\mathclose{[}$, there exists a strictly increasing solution of \eqref{bvp-1}.
\end{theorem}

\begin{proof}
Let $\rho\in\mathopen{]}0,\alpha\mathclose{[}$ and $T>0$. 
Theorem~\ref{th-exist-mixed-1} ensures the existence of a strictly increasing solution $v_{T}$ of \eqref{mixed-1} in the bounded interval $\mathopen{[}t_{0}-T,t_{0}\mathclose{]}$; the goal is now to 
pass to the limit for $T\to+\infty$. To this end, we first prove that $v_{T}$, $v_{T}'$, $v_{T}''$ are uniformly bounded in $\mathopen{[}t_{0}-T,t_{0}\mathclose{]}$. Indeed, we observe that
\begin{equation*}
v_{T}(t)\in\mathopen{]}0,\alpha\mathclose{[}, \quad \text{for all $t\in \mathopen{[}t_{0}-T,t_{0}\mathclose{]}$.}
\end{equation*}
Moreover, we have
\begin{equation*}
0\leq v_{T}''(t) = -q(t) f(v_{T}(t)) \bigl{(} 1- (v_{T}'(t))^{2} \bigr{)}^{\frac{3}{2}} \leq \|q\|_{L^{\infty}(-\infty,t_{0})} \max_{\mathopen{[}0,\alpha\mathclose{]}}|f| =:M,
\end{equation*}
for almost every $t\in \mathopen{[}t_{0}-T,t_{0}\mathclose{]}$.
Recalling the discussion in the Introduction and in Section~\ref{section-2}, we then consider the first-order reduction associated with 
\eqref{eq-sect-3}, reading as 
\begin{equation}\label{eq-intermedia1}
\dot{y}(v) = -q(t(v)) \frac{f(v)}{\delta}.
\end{equation}
Integrating such an equality 
between $v_{T}(t_{0}-T)$ (where $y(v_{T}(t_{0}-T))=0$) and $v\in\mathopen{]}v_{T}(t_{0}-T),\rho\mathclose{[}$, we deduce that
\begin{equation*}
y(v)= - \int_{v_{T}(t_{0}-T)}^{v} q(t(s)) \frac{f(s)}{\delta} \,\mathrm{d}s,
\end{equation*}
and thus
\begin{equation*}
|y(v)| \leq \frac{\alpha M}{\delta}, \quad \text{for all $v\in\mathopen{]}v_{T}(t_{0}-T),\rho\mathclose{[}$.}
\end{equation*}
Then, recalling the latter equality in \eqref{eq-first}, we deduce that there exists $K>0$ such that
\begin{equation}\label{bound-deriv}
v_{T}'(t)\in\mathopen{[}0,K\mathclose{]}, \quad \text{for all $t\in \mathopen{[}t_{0}-T,t_{0}\mathclose{]}$.}
\end{equation}
Consequently, there exists a continuously differentiable function $v_{\infty}$ such that $v_{T} \to v_{\infty}$ and $v_{T}' \to v_{\infty}'$ for $T\to +\infty$, with uniform convergence. Using \eqref{eq-sect-3}, also $v_{T}'' \to v_{\infty}''$ almost everywhere for $T \to +\infty$. We deduce that $v_{\infty}$ is a strictly increasing solution of equation \eqref{eq-sect-3} on the interval $\mathopen{]}-\infty, 0\mathclose{[}$. It remains to prove that $v_{\infty}(t_{0})=\rho$ and $v_{\infty}(-\infty)=0$, thus proving that $v_{\infty}$ solves \eqref{bvp-1}.

The former claim immediately follows from the fact that $v_{T}(t_{0})=\rho$ for all $T>0$. In order to prove that $v_{\infty}(-\infty)=0$, we observe that $v_{T}(t_{0}-T)\in\mathopen{]}0,\rho\mathclose{[}$ for all $T>0$; hence, passing to a subsequence if necessary, there exists $\ell\in\mathopen{[}0,\rho\mathclose{]}$ for which $\lim_{k\to-\infty} v_{\infty}(k) = \ell$. By contradiction, assume that $\ell > 0$; then, 
since $0<\ell\leq v_{\infty}(t)\leq \rho < \alpha$ and $v_{\infty}'(t)> 0$ for every $t \in\mathopen{]}-\infty,t_{0}\mathclose{]}$, recalling \eqref{bound-deriv} we have
\begin{align*}
\phi(K) \geq \phi(v_{\infty}'(t_{0})) \geq \phi(v_{\infty}'(t_{0})) - \phi(v_{\infty}'(t)) 
&= - \int_{t}^{t_{0}} q(\xi) \frac{f(v_{\infty}(\xi))}{\delta}\,\mathrm{d}\xi
\\
&\geq \|q\|_{L^{1}(t,t_{0})} \frac{\min_{\mathopen{[}\ell,\rho\mathclose{]}} (-f)}{\delta} >0,
\end{align*}
for every $t\in\mathopen{]}-\infty,t_{0}\mathclose{]}$.
The fact that $\|q\|_{L^{1}(t,t_{0})}\to+\infty$ for $t \to -\infty$ then yields a contradiction. Hence, $\ell =0$ and 
the proof is complete.
\end{proof}

Proceeding in an analogous way, integrating in particular \eqref{eq-intermedia1} between $v_T(t_{0}+T)$ and $v \in \mathopen{]}\rho, v_T(t_{0}+T)\mathclose{[}$, where $\rho \in \mathopen{]} \beta, 1 \mathclose{[}$ is fixed, one can prove the following.

\begin{theorem}\label{th-exist-bvp-2}
Let $\delta > 0$ and $t_{0}\in\mathbb{R}$. Moreover, let $q\in L^{\infty}(t_{0},+\infty)$ be a positive weight such that $\|q\|_{L^{1}(t_{0},t)}\to+\infty$ as $t\to+\infty$ and assume that $f\colon\mathopen{[}0, 1\mathclose{]} \to \mathbb{R}$ is a Lipschitz continuous function fulfilling \ref{hp-f-1}. 
Then, for every $\rho\in\mathopen{]}\beta,1\mathclose{[}$, there exists a strictly increasing solution of \eqref{bvp-2}.
\end{theorem}

\subsection{Heteroclinic and homoclinic solutions}\label{section-3.3}

We now exploit the results of the previous sections to construct heteroclinic and homoclinic solutions. We deal with a bistable reaction term, providing the existence of a strictly increasing heteroclinic solution $v$ of \eqref{eq-sect-3} with $v(t_{0}) \in\mathopen{]}0,\alpha\mathclose{[}$. Some remarks for the more general case $\alpha<\beta$ are given as well.

\begin{theorem}\label{th-main}
Let $\delta > 0$ and let $q\in L^{\infty}(-\infty,t_{0})$ for some $t_{0} \in \mathbb{R}$. Assume that $q$ satisfies the following two assumptions:
\begin{itemize}
\item there exists $\eta>0$ for which $q(t)\geq\eta$ for almost every $t\in\mathopen{]}-\infty,t_{0}\mathclose{[}$;
\item there exists $c > 0$ for which $q\equiv c$ in $\mathopen{[}t_{0},+\infty\mathclose{[}$.
\end{itemize} 
Moreover, let $f$ be a Lipschitz continuous function satisfying \ref{hp-f-1} and \ref{hp-f-2} and assume that $f$ is bistable, that is, $\alpha=\beta$. Then, if 
\begin{equation}\label{cond-c}
c \leq \eta \dfrac{-F(\alpha)}{F(1)-F(\alpha)},
\end{equation}
there exists a strictly increasing heteroclinic solution of \eqref{eq-sect-3}. 
\end{theorem}

\begin{proof}
First, let us focus our attention on the interval $\mathopen{]}-\infty,t_{0}\mathclose{]}$. Since $\|q\|_{L^{1}(t,t_{0})}\to+\infty$ as $t\to-\infty$, for each $\rho \in \mathopen{]}0,\alpha\mathclose{[}$ we can consider the strictly increasing solution $v_{\infty}^{\rho}$ of \eqref{bvp-1} provided by Theorem~\ref{th-exist-bvp-1}, and we set $\kappa (\rho):= \phi(v_{\infty}^{\rho}{}'(t_{0}))$. Accordingly, we define
\begin{equation}\label{eq-defyinfty}
y_{\infty}^{\rho} (v) := \dfrac{1}{\sqrt{1-(v_{\infty}^{\rho}{}'(t(v)))^{2}}}-1, 
\end{equation}
where $v\mapsto t(v)$ is the inverse function of $t\mapsto v_{\infty}^{\rho}(t)$.
From the discussion in the Introduction and in Section~\ref{section-2}, $y_{\infty}^{\rho}$ is a solution of the first-order equation
\begin{equation*}
\dot{y}(v) = -\frac{q(t(v))}{\delta} f(v), \quad v\in\mathopen{[}0,\rho\mathclose{]}.
\end{equation*}
Hence, recalling that $f \leq 0$ on $\mathopen{[}0, \rho\mathclose{]}$, we have
\begin{equation*}
-\frac{\eta}{\delta} f(v)
\leq \dot{y}_{\infty}^{\rho}(v) \leq 
-\frac{\|q\|_{L^{\infty}(-\infty,t_{0})}}{\delta} f(v) ,
\quad \text{for all $v\in\mathopen{[}0,\rho\mathclose{]}$.}
\end{equation*}
By the definition of $F$ and the fact that $y_{\infty}^{\rho}(0)=0$, integrating on $\mathopen{[}0,\rho\mathclose{]}$ gives
\begin{equation}\label{estimate-y}
-\frac{\eta}{\delta} F(\rho) \leq y_{\infty}^{\rho}(\rho) \leq -\frac{\|q\|_{L^{\infty}(-\infty,t_{0})}}{\delta} F(\rho).
\end{equation}
Next, observing that
\begin{align*}
\phi(v_{\infty}^{\rho}{}'(t)) = \sqrt{(y_{\infty}^{\rho}(v_{\infty}^{\rho}(t)))^{2}+2y_{\infty}^{\rho}(v_{\infty}^{\rho}(t))}, 
\end{align*}
and since the function $s\mapsto\sqrt{s^{2}+2s}$ is strictly increasing, from \eqref{estimate-y} we deduce that
\begin{equation}\label{kappa-rho-estimate}
\sqrt{\frac{\eta^{2}}{\delta^2}F(\rho)^{2}-2\frac{\eta}{\delta} F(\rho)} 
\leq \kappa(\rho) \leq 
\sqrt{\frac{\|q\|_{L^{\infty}(-\infty,t_{0})}^{2}}{\delta^2}F(\rho)^{2}-2\frac{\|q\|_{L^{\infty}(-\infty,t_{0})}}{\delta}F(\rho)}.
\end{equation}
Finally, it is straightforward that
\begin{equation}\label{rho-0}
\lim_{\rho\to0^{+}} (\rho, \kappa(\rho)) = (0,0).
\end{equation}

Second, let us consider the interval $\mathopen{[}t_{0},+\infty\mathclose{[}$, where the equation is autonomous.
Defining $\mathcal{E}(v,w)$ as in \eqref{def-energy},
the solutions $(v,w)$ of the associated planar system \eqref{syst-sect-3} whose orbit lies on the level line $\mathcal{E}(v,w)=\mathcal{E}(1,0)$ satisfy
\begin{equation*}
\sqrt{1+w^{2}} - 1 + \frac{c}{\delta} F(v) = \frac{c}{\delta} F(1).
\end{equation*}
Therefore,
\begin{equation}\label{curve-c2}
w = w(v) = \sqrt{\frac{c^{2}}{\delta^2}(F(1)-F(v))^{2} + 2 \frac{c}{\delta} (F(1)-F(v))}.
\end{equation}
In order to prove the existence of a heteroclinic solution, we have to show that the two parametric curves $(v, \kappa(v))$ and $(v, w(v))$ intersect. To this end, we show that the function $v \mapsto w(v)-\kappa(v)$ changes sign at least once in $\mathopen{]}0, \alpha\mathclose{[}$. Due to \eqref{rho-0} and the fact that $w(0)=\sqrt{\frac{c^{2}}{\delta^2}(F(1))^{2} + 2 \frac{c}{\delta} F(1)} >0$ (since $F(1)>0$), we have that $\lim_{v \to 0^{+}} (w(v)-\kappa(v)) > 0$. On the other hand, we prove that \eqref{cond-c} implies that $w(v)-\kappa(v) < 0$ for some $v \in \mathopen{]}0, \alpha\mathclose{[}$. 
By contradiction, assume that $w(v) - \kappa(v) > 0$ for every $v \in \mathopen{]}0, \alpha\mathclose{[}$. Then, thanks to \eqref{kappa-rho-estimate},
\begin{equation*}
\sqrt{\frac{c^{2}}{\delta^2}(F(1)-F(v))^{2} + 2 \frac{c}{\delta} (F(1)-F(v))} =w(v) > \kappa(v) \geq \sqrt{\frac{\eta^{2}}{\delta^2} F(v)^{2}-2\frac{\eta}{\delta} F(v)}.
\end{equation*}
Using again the fact that the function $s\mapsto\sqrt{s^{2}+2s}$ is strictly increasing, for every $v\in\mathopen{]}0,\alpha\mathclose{[}$ one then has
\begin{equation*}
\frac{c}{\delta} (F(1)-F(v)) > -\frac{\eta}{\delta} F(v), 
\end{equation*}
whence
\begin{equation}\label{eq-contr}
c > \eta \dfrac{-F(v)}{F(1)-F(v)}, 
\end{equation}
a contradiction with \eqref{cond-c}, since the right-hand side in \eqref{eq-contr} is monotone in $v$.  The proof is complete.
\end{proof}

\begin{remark}\label{rem-3.2}
Since 
\begin{equation*}
\sup_{v\in\mathopen{[}0,\alpha\mathclose{]}} \frac{-F(v)}{F(1)-F(v)}=\frac{-F(\alpha)}{F(1)-F(\alpha)},
\end{equation*}
condition \eqref{cond-c} is sufficient for the above argument. Moreover, in case $q$ is constant in $\mathopen{]}-\infty,t_{0}\mathclose{[}$, it produces the fourth alternative in the statement of Proposition~\ref{prop-intro}, allowing one to recover the existence result for the problem with a two-step weight $q$.
\hfill$\lhd$
\end{remark}

The analog for the case when $q$ is definitively constant at $-\infty$, involving problem \eqref{bvp-2}, provides a strictly increasing heteroclinic solution $v$ of \eqref{eq-sect-3} with $v(t_{0}) \in\mathopen{]}\alpha,1\mathclose{[}$ and can be formulated as follows. We give an outline of the proof for the reader's convenience.

\begin{theorem}\label{th-main-b}
Let $\delta > 0$ and let $q\in L^{\infty}(t_{0}, +\infty)$ for some $t_{0} \in \mathbb{R}$. Assume that $q$ fulfills the following two assumptions:
\begin{itemize}
\item there exists $\eta>0$ for which $q(t)\geq\eta$ for almost every $t\in\mathopen{]}t_{0}, +\infty\mathclose{[}$;
\item there exists $c >0$ for which $q\equiv c$ in $\mathopen{]}-\infty, t_{0}\mathclose{]}$.
\end{itemize} 
Moreover, let $f$ be a Lipschitz continuous function satisfying \ref{hp-f-1} and \ref{hp-f-2} and assume that $f$ is bistable, that is, $\alpha=\beta$. Then, if 
\begin{equation}\label{cond-c2}
c \geq \Vert q \Vert_{L^{\infty}(t_{0}, +\infty)} \dfrac{F(1)-F(\alpha)}{-F(\alpha)},
\end{equation}
there exists a strictly increasing heteroclinic solution of \eqref{eq-sect-3}.
\end{theorem}

\begin{proof}
Similarly as in the proof of Theorem~\ref{th-main}, one first considers the interval $\mathopen{]}t_{0},+\infty\mathclose{]}$, working with the strictly increasing solution $v_{\infty}^{\rho}$ provided by Theorem~\ref{th-exist-bvp-2}. Defining $y_\infty^\rho$ as in \eqref{eq-defyinfty}, recalling the positive sign of $f$ in $[\rho, 1]$ and the fact that $y_\infty^\rho(1)=0$ one then has 
\begin{equation*}
\frac{\eta}{\delta} (F(1)-F(\rho)) \leq y_{\infty}^{\rho}(\rho) \leq \frac{\|q\|_{L^{\infty}(t_{0},+\infty)}}{\delta} (F(1)-F(\rho)), 
\end{equation*}
yielding
\begin{align*}
&\sqrt{\frac{\eta^{2}}{\delta^2}(F(1)-F(\rho))^{2}-2\frac{\eta}{\delta} (F(1)-F(\rho))} 
\leq \kappa(\rho) 
\\
&\leq 
\sqrt{\frac{\|q\|_{L^{\infty}(t_{0},+\infty)}^{2}}{\delta^2}(F(1)-F(\rho))^{2}-2\frac{\|q\|_{L^{\infty}(t_{0},+\infty)}}{\delta}(F(1)-F(\rho))},
\end{align*}
where $\kappa (\rho):= \phi(v_{\infty}^{\rho}{}'(t_{0}))$.
On the other hand, the energy curve corresponding to the solutions $(v,w)$ of \eqref{syst-sect-3} whose orbit emanates from $
(0, 0)$ is given by 
\begin{equation*}
\sqrt{1+w^{2}} - 1 + \frac{c}{\delta} F(v) = 0, 
\quad \text{ namely } \quad 
w(v) = \sqrt{\frac{c^{2}}{\delta^2}F(v)^{2} + 2 \frac{c}{\delta} F(v)},
\end{equation*}
making sense only for $v \in [0, v_{0}]$. 
Now, the function $v \mapsto w(v)-\kappa(v)$ is such that $w(v_{0})-\kappa(v_{0}) < 0$, hence it suffices to show that $\kappa(v) < w(v)$ for some $v \in \mathopen{]} \beta, v_{0} \mathclose{[}$ in order to prove the statement. If by contradiction it were $w(v) - \kappa(v) < 0$ for every $v \in \mathopen{]} \beta, v_{0} \mathclose{[}$, thanks to the above estimates this would lead to 
\begin{equation*}
-\frac{c}{\delta} F(v) < \frac{\|q\|_{L^{\infty}(t_{0},+\infty)}}{\delta} (F(1)-F(v)) 
\end{equation*}
(recall that $F(v) < 0$ for $v < v_{0}$), 
whence
\begin{equation}\label{eq-aux1}
c < \|q\|_{L^{\infty}(t_{0},+\infty)} \dfrac{F(1)-F(v)}{-F(v)}
\end{equation}
for every $v \in \mathopen{]} \beta, v_{0} \mathclose{[}$, 
a contradiction with \eqref{cond-c2} since the right-hand side in \eqref{eq-aux1} is monotone in $v$.
\end{proof}

\begin{remark}\label{rem-3.3}
If $\alpha \neq \beta$, the argument in the proofs of Theorem~\ref{th-main} and Theorem~\ref{th-main-b} still works but will not provide, in general, an increasing heteroclinic. However, as for Theorem~\ref{th-main}, one will find a heteroclinic which is definitively increasing both at $-\infty$ (thanks to the construction in Section \ref{section-3.2}) and at $+\infty$ (due to the positive sign of $f$ in a left neighborhood of $1$), possibly displaying a certain number of monotonicity changes in between, according to the number of sign changes of $F(1)-F(v)$. The picture for Theorem~\ref{th-main-b} is of course reversed.
\hfill$\lhd$
\end{remark}

Concerning homoclinics, one can carry out a similar argument in order to intersect, in the upper phase-plane, the orbit corresponding to the solution of \eqref{bvp-1} with the solution of \eqref{syst-sect-3} passing through the point $(v_{0}, 0)$ (which for $t \to +\infty$ converges to $(0, 0)$). One then simply has to replace $F(1)$ with $F(v_{0})$ (which is equal to $0$) in the statement of Theorem~\ref{th-main}, yielding the existence of a homoclinic solution $v$ of \eqref{eq-sect-3}, taking $v_{0}$ as maximum value, with $v(t_{0})\in\mathopen{[}0,\alpha\mathclose{[}$. For simplicity, we only give the statement in case $q$ is definitively constant at $+\infty$.

\begin{theorem}\label{th-main-h}
Let $\delta > 0$ and let $q\in L^{\infty}(-\infty,t_{0})$ for some $t_{0} \in \mathbb{R}$. Assume that $q$ fulfills the following two assumptions:
\begin{itemize}
\item there exists $\eta>0$ for which $q(t)\geq\eta$ almost everywhere in $\mathopen{]}-\infty,t_{0}\mathclose{[}$;
\item there exists $c > 0$ for which $q\equiv c$ in $\mathopen{[}t_{0},+\infty\mathclose{[}$.
\end{itemize} 
Moreover, let $f$ be a Lipschitz continuous function satisfying \ref{hp-f-1} and \ref{hp-f-2} and assume that $f$ is bistable, that is, $\alpha=\beta$. Then, if $c \leq \eta$, 
there exists a homoclinic solution $v$ of \eqref{eq-sect-3} taking $v_{0}$ as maximum value.
\end{theorem}

Similarly, if one replaced the branch through $(v_{0}, 0)$ with the one through $(w, 0)$ for $w < v_{0}$, one would find solutions starting increasingly at $-\infty$ and then being definitively periodic, in line with the shape of the solutions of system \eqref{syst-sect-3} with a stepwise constant weight $q$.

\begin{theorem}\label{th-main-h-2}
Let $\delta > 0$ and let $q\in L^{\infty}(-\infty,t_{0})$ for some $t_{0} \in \mathbb{R}$. Assume that $q$ fulfills the following two assumptions:
\begin{itemize}
\item there exists $\eta>0$ for which $q(t)\geq\eta$ almost everywhere in $\mathopen{]}-\infty,t_{0}\mathclose{[}$;
\item there exists $c > 0$ for which $q\equiv c$ in $\mathopen{[}t_{0},+\infty\mathclose{[}$.
\end{itemize} 
Moreover, let $f$ be a Lipschitz continuous function satisfying \ref{hp-f-1} and \ref{hp-f-2} and assume that $f$ is bistable, that is, $\alpha=\beta$.
Then, if there exists $w < v_0$ such that 
\begin{equation}\label{cond-c2-h2}
c \leq \eta\dfrac{-F(\alpha)}{F(w)-F(\alpha)},
\end{equation} 
there exists a solution $v$ of \eqref{eq-sect-3} satisfying $v(-\infty)=0$ which is definitively periodic.
\end{theorem}

Condition~\eqref{cond-c2-h2} ensures that the two  parametric curves $(v, \kappa(v))$ and $(v, w(v))$ defined in the proof of Theorem~\ref{th-main} intersect (with the only difference that the latter one emanates from $(w, 0)$, with $w < v_0$, rather than from $(v_0, 0)$). Of course, the possibility that these two curves always intersect for $t=t_0$ is due to the possibility of shifting the time along the latter branch. As before, condition \eqref{cond-c2-h2} is independent of the value of $\delta$. 
Similar statements to Theorem~\ref{th-main-h} and Theorem~\ref{th-main-h-2} can be given in case $q$ is definitively constant at $-\infty$, but we omit the details for briefness.

\smallskip

We now turn to nonexistence. Here, it is natural to expect that if $c$ is sufficiently large (respectively, small) then the intersection argument used in the statement of Theorem~\ref{th-main} (respectively, Theorem~\ref{th-main-b}) will not hold. To prove this claim, we first give a necessary condition on $\rho$ for \eqref{bvp-1} to be solvable. 

\begin{lemma}\label{lem-stimapriori}
Let $q$ and $f$ fulfill the assumptions of Theorem~\ref{th-main}. Moreover, assume that
\begin{equation}\label{cond-eta}
\eta>\Vert q \Vert_{L^{\infty}(-\infty, t_{0})} \frac{-F(\alpha)}{F(1)-F(\alpha)}.
\end{equation}
Then, there exists $M < 1$ such that, if \eqref{bvp-1} has a solution, then $\rho \leq M$.
\end{lemma}

\begin{proof}
By contradiction, let us assume that there exist $\rho_n \to 1$ and a sequence of solutions $v_n$ of \eqref{bvp-1} such that $v_n(t_{0})=\rho_n$. Then, there exist $\hat{t}_n \in \mathopen{[}-\infty,t_{0}\mathclose{[}$ and $\check{t}_n \in \mathopen{]}-\infty,t_{0}\mathclose{]}$, with $\hat{t}_n < \check{t}_n$, such that $v_n'(\hat{t}_n)=0$, $v_n(\check{t}_n)=\rho_n$ and $v_n' > 0$ on $\mathopen{]}\hat{t}_n,\check{t}_n\mathclose{[}$. Defining 
$y_n(v)= 1/\sqrt{1-(v_n'(t(v)))^{2}}-1$,
it follows that $y_n(v_n(\hat{t}_n))=0$ and hence, integrating on $\mathopen{]}\hat{t}_n,\check{t}_n\mathclose{[}$, one has
\begin{equation*}
y_n(v_n(\check{t}_n)) = \int_{v_n(\hat{t}_n)}^{v_n(\check{t}_n)} \frac{-q(t(s))}{\delta} f(s) \, \mathrm{d}s. 
\end{equation*}
If $v_n(\hat{t}_n) \geq \alpha$, this is trivially a contradiction, since the right-hand side is strictly negative while the left-hand one is positive. Otherwise, from the above we have 
\begin{align*}
y_n(v_n(\check{t}_n)) & \leq \int_{v_n(\hat{t}_n)}^{\alpha} \frac{-q(t(s))}{\delta} f(s) \, \mathrm{d}s + \int_{\alpha}^{\rho_n} \frac{-q(t(s))}{\delta} f(s) \, \mathrm{d}s \\ 
& \leq \frac{1}{\delta} \big(-\Vert q \Vert_{L^{\infty}(-\infty, t_{0})} F(\alpha) - \eta (F(\rho_n)-F(\alpha)) \big).
\end{align*}
Using \eqref{cond-eta}, we then get the same sign contradiction, proving the statement.
\end{proof}

We remark that \eqref{cond-eta} implies that
\begin{equation*}
\eta \leq q(t) <  \eta \dfrac{F(1)-F(\alpha)}{-F(\alpha)}, \quad \text{for every $t\in\mathopen{]}-\infty,t_{0}\mathclose{[}$,}
\end{equation*}
thus the smaller $\eta$, the smaller the oscillation allowed for $q$.

The nonexistence result then reads as follows.

\begin{theorem}\label{thm-nonesistenza}
Let $q$ and $f$ fulfill the assumptions of Theorem~\ref{th-main}. Then, if \eqref{cond-eta} and
\begin{equation}\label{cond-c3}
c > \| q \|_{L^{\infty}(-\infty,t_{0})}\frac{- F(\alpha) }{F(1)-F(\alpha)}
\end{equation}
hold, no heteroclinic solutions of \eqref{eq-sect-3} exist. 
\end{theorem}

\begin{proof}
Assume by contradiction that there exists a heteroclinic solution $v$; let $\rho=v(t_{0})$ and denote by $v_{\infty}^{\rho,-}$ and $v_{\infty}^{\rho,+}$, respectively, the restrictions of $v$ to the intervals $\mathopen{]}-\infty,t_{0}\mathclose{]}$ and $\mathopen{]}t_{0}, +\infty\mathclose{[}$. Since $v_{\infty}^{\rho,-}$ is a solution of \eqref{bvp-1}, by Lemma~\ref{lem-stimapriori} one has $\rho < M$. Our aim is now to prove that the set of points $\{(\rho, v_{\infty}^{\rho,-}{}'(t_{0})) \colon \rho \in \mathopen{[}0,M\mathclose{]}\}$ and the level set $\{(v, w) \colon \mathcal{E}(v, w)=\mathcal{E}(1, 0)\}$ for the energy $\mathcal{E}$ defined in \eqref{def-energy} (which depends on $c$) cannot intersect (in the upper phase-plane $\{\phi(v') > 0\}$) for $c$ sufficiently large, implying a contradiction. We will achieve this goal by showing that \eqref{cond-c3} implies $v_{\infty}^{\rho,-}{}'(t_{0}) < v_{\infty}^{\rho,+}{}'(t_{0})$ for any $\rho \in \mathopen{[}0,M\mathclose{]}$. 

Thus, let $\rho \in \mathopen{[}0,M\mathclose{]}$ be fixed. The conclusion trivially holds if $v_{\infty}^{\rho,-}{}'(t_{0}) \leq 0$, since $v_{\infty}^{\rho,+}{}'(t_{0}) > 0$ thanks to \eqref{curve-c2}. Hence, we can assume $v_{\infty}^{\rho,-}{}'(t_{0}) > 0$; being $v_{\infty}^{\rho,-}{}' > 0$ in a left neighborhood $\mathopen{]}\hat{t},t_{0}\mathclose{[}$ of $t_{0}$, we can repeat the argument in the previous lemma (with $\hat{t}_n=\hat{t}$ and $\check{t}_n=t_{0}$) to obtain 
\begin{equation*}
y_\infty^\rho(\rho) \leq \frac{1}{\delta} \big(-\| q \|_{L^{\infty}(-\infty,t_{0})} F(\alpha) - \eta (F(\rho)-F(\alpha)) \big) =:\mathcal{B}(\rho),
\end{equation*}
where $y_\infty^\rho$ is defined as in \eqref{eq-defyinfty}. 
Since 
\begin{align*}
\phi(v_{\infty}^{\rho,-}{}'(t_{0})) = \sqrt{(y_{\infty}^\rho(\rho))^{2}+2y_{\infty}^\rho(\rho)} \leq \sqrt{\mathcal{B}(\rho)^{2}+2\mathcal{B}(\rho)} , 
\end{align*}
we will obtain a contradiction if it holds
\begin{equation*}
\sqrt{\mathcal{B}(\rho)^{2}+2\mathcal{B}(\rho)} < \sqrt{\frac{c^{2}}{\delta^2}(F(1)-F(\rho))^{2} + 2 \frac{c}{\delta} (F(1)-F(\rho))} = \phi(v_{\infty}^{\rho,-}{}'(t_{0})).
\end{equation*}
As $s\mapsto\sqrt{s^{2}+2s}$ is strictly increasing, this is equivalent to 
\begin{equation}\label{stima-c-grande}
c > \dfrac{-\| q \|_{L^{\infty}(-\infty,t_{0})} F(\alpha) - \eta(F(\rho)-F(\alpha))}{F(1)-F(\rho)} := \mathcal{C}(\rho).
\end{equation}
Since
\begin{equation*}
\mathcal{C}'(\rho) = -\dfrac{f(\rho) \bigl{(} \eta (F(1)-F(\alpha)) +\| q \|_{L^{\infty}(-\infty,t_{0})} F(\alpha) \bigr{)}}{\bigl{(}F(1)-F(\rho)\bigr{)}^{2}},
\end{equation*}
conditions \eqref{cond-eta} and \ref{hp-f-1} ensure that $\mathcal{C}'(\rho)>0$ for every $\rho\in\mathopen{]}0,\alpha\mathclose{[}$, and $\mathcal{C}'(\rho)<0$ for every $\rho\in\mathopen{]}\alpha,M\mathclose{]}$. As a consequence, $\max_{\rho\in\mathopen{[}0,M\mathclose{]}} \mathcal{C}(\rho) = \mathcal{C}(\alpha)$ and \eqref{cond-c3} implies \eqref{stima-c-grande}, concluding the proof.
\end{proof}

\begin{remark}\label{rem-3.3b}
If $M<\alpha$, the bound in \eqref{cond-c3} can actually be improved since the above argument holds in the same way for
\begin{equation*}
c > \frac{-\| q \|_{L^{\infty}(-\infty,t_{0})} F(\alpha) - \eta(F(M)-F(\alpha))}{F(1)-F(M)}.
\end{equation*}
Moreover, notice that in the stepwise constant case $q \equiv c_1$ on $\mathopen{]}-\infty,t_{0}\mathclose{]}$, $q \equiv c_2$ on $\mathopen{]}t_{0}, +\infty\mathclose{]}$ (as in Proposition~\ref{prop-intro}), one has that $M=v_{0}$ (where $M$ is as in Lemma~\ref{lem-stimapriori}) and, taking into account that $\eta=c_1$, condition \eqref{cond-c3} becomes $c_2 > - c_1 F(\alpha) / ( F(1)-F(\alpha))$, in accord with the result stated in Proposition~\ref{prop-intro}. 
At last, we point out that it would be more difficult to provide the nonexistence result in the case $\alpha<\beta$, since a precise knowledge of the sign of $f$ seems essential to carry out a non-intersection argument.
\hfill$\lhd$
\end{remark}

One can deal similarly with the case when $q$ is definitively constant at $-\infty$, first proving that there exists $M <1$ such that, if \eqref{bvp-2} has a solution, then $\rho > M$, and next using a similar argument as the one in the proof of Theorem~\ref{thm-nonesistenza}. In the same way one can also deal with homoclinics, this time considering the autonomous branch emanating, in the phase-plane, from the point $(v_{0}, 0)$. We omit these statements for briefness.

\begin{figure}[htb]
\centering
\begin{subfigure}[t]{0.3\textwidth}
\centering
\begin{tikzpicture}[scale=1]
\begin{axis}[
  scaled ticks=false,
  tick label style={font=\scriptsize},
  axis y line=left, axis x line=middle,
  xtick={0.45, 1},
  ytick={0},
  xticklabels={$\alpha$, $1$},
  yticklabels={$0$},
  xlabel={\small $v$},
  ylabel={\small $\phi(v')$},
every axis x label/.style={  at={(ticklabel* cs:1.0)},  anchor=west,
},
every axis y label/.style={  at={(ticklabel* cs:1.0)},  anchor=south,
},
  width=5cm,
  height=5cm,
  xmin=0,
  xmax=1.1,
  ymin=-1,
  ymax=1]
\addplot [color=white, fill=magenta, fill opacity=0.1, line width=0pt, smooth] coordinates {(0., 0.) (0.1,0.2) (0.2, 0.6) (0.3, 0.6) (0.45,0.95) (0.45,0.75) (0.3,0.45) (0.2, 0.35) (0.1,0.01) (0.,0.)};
\addplot [color=white, fill=white, line width=0pt] coordinates {(0.45, 0.01) (0.6, 0.01) (0.6, 1.) (0.45, 1.) (0.45,0.01)};
\addplot [color=magenta,line width=0.7pt,smooth] coordinates {(0., 0.) (0.1,0.05) (0.2,0.5) (0.3,0.5) (0.45,0.9)};
\addplot [color=cyan,line width=0.7pt,smooth] coordinates {(1., 0) (0.992078, 0.0186396) (0.991379, 0.0192703) (0.985634, 0.0344828) (0.984032, 0.0370299) (0.982759, 0.0390545) (0.977859, 0.0517241) (0.975925, 0.0552984) (0.970729, 0.0689655) (0.967755, 0.0734406) (0.965517, 0.0769574) (0.95952, 0.0914539) (0.954618, 0.103448) (0.95122, 0.109336) (0.948276, 0.114497) (0.942853, 0.127085) (0.938152, 0.137931) (0.934418, 0.144697) (0.931034, 0.150805) (0.925913, 0.162171) (0.921226, 0.172414) (0.917338, 0.179504) (0.913793, 0.185889) (0.908691, 0.196693) (0.903768, 0.206897) (0.899971, 0.213734) (0.896552, 0.219778) (0.891175, 0.230625) (0.885704, 0.241379) (0.882302, 0.247362) (0.87931, 0.252502) (0.87335, 0.26394) (0.866956, 0.275862) (0.864317, 0.280357) (0.862069, 0.284087) (0.8552, 0.296608) (0.847427, 0.310345) (0.845999, 0.312688) (0.844828, 0.314556) (0.83671, 0.328592) (0.830557, 0.338886) (0.827586, 0.343841) (0.827326, 0.344307) (0.827025, 0.344828) (0.817839, 0.359816) (0.810345, 0.37162) (0.808254, 0.375128) (0.805661, 0.37931) (0.798567, 0.390238) (0.793103, 0.398351) (0.788776, 0.405139) (0.783035, 0.413793) (0.778878, 0.419825) (0.775862, 0.424213) (0.771178, 0.431034) (0.768868, 0.434288) (0.767241, 0.436578) (0.758923, 0.448276) (0.758743, 0.448521) (0.758621, 0.448681) (0.757981, 0.449555) (0.748468, 0.462453) (0.741379, 0.471655) (0.738064, 0.476127) (0.732914, 0.482759) (0.724138, 0.493622) (0.716858, 0.502682) (0.710263, 0.510509) (0.706897, 0.514454) (0.706028, 0.515504) (0.704513, 0.517241) (0.695017, 0.527966) (0.689655, 0.533726) (0.683853, 0.540119) (0.672747, 0.551724) (0.672527, 0.551951) (0.672414, 0.552062) (0.672076, 0.5524) (0.660967, 0.563314) (0.655172, 0.56889) (0.655092, 0.568966) (0.649227, 0.574317) (0.639574, 0.58292) (0.637931, 0.584372) (0.637288, 0.58492) (0.635777, 0.586207) (0.62931, 0.591599) (0.625082, 0.594991) (0.62069, 0.598386) (0.612674, 0.604658) (0.610767, 0.606051) (0.603448, 0.611226) (0.599961, 0.613716) (0.589437, 0.62069) (0.586207, 0.622786) (0.584692, 0.623719) (0.573716, 0.630191) (0.568966, 0.632764) (0.560677, 0.637266) (0.560149, 0.637539) (0.55935, 0.637931) (0.551724, 0.641608) (0.546175, 0.644074) (0.538392, 0.647354) (0.534483, 0.648883) (0.531852, 0.64991) (0.517376, 0.654904) (0.517241, 0.654948) (0.517143, 0.654976) (0.51641, 0.655172) (0.501936, 0.659044) (0.5, 0.659481) (0.497576, 0.66002) (0.486251, 0.662158) (0.482759, 0.662676) (0.478714, 0.663262) (0.470039, 0.664216) (0.465517, 0.664536) (0.460672, 0.664864) (0.453238, 0.665096) (0.448276, 0.665066) (0.443352, 0.665021) (0.435783, 0.664669) (0.431034, 0.664279) (0.426665, 0.66391) (0.417606, 0.662798) (0.413793, 0.662201) (0.410536, 0.661686) (0.398636, 0.659342) (0.396552, 0.658897) (0.394896, 0.658485) (0.382569, 0.655172) (0.379697, 0.6544) (0.37931, 0.654288) (0.378783, 0.654118) (0.364935, 0.64944) (0.362069, 0.648351) (0.357871, 0.646777) (0.350509, 0.64381) (0.344828, 0.641412) (0.337139, 0.637931) (0.336382, 0.637581) (0.335922, 0.63736) (0.327586, 0.633142) (0.322598, 0.630666) (0.31266, 0.625321) (0.310345, 0.62402) (0.309039, 0.623302) (0.30461, 0.62069) (0.29574, 0.615417) (0.293103, 0.613735) (0.288064, 0.610611) (0.282641, 0.607132) (0.275862, 0.602496) (0.269744, 0.598444) (0.262027, 0.593019) (0.258621, 0.590604) (0.257008, 0.589433) (0.252624, 0.586207) (0.244436, 0.580093) (0.241379, 0.57778) (0.234528, 0.572505) (0.231976, 0.570531) (0.224138, 0.564163) (0.219656, 0.560688) (0.208616, 0.551724) (0.2074, 0.550717) (0.206897, 0.550279) (0.205669, 0.549269) (0.195248, 0.540539) (0.189655, 0.535649) (0.183134, 0.530285) (0.175798, 0.52401) (0.172414, 0.521071) (0.171052, 0.519966) (0.167807, 0.517241) (0.158991, 0.509603) (0.155172, 0.506172) (0.146897, 0.49931) (0.14569, 0.498277) (0.137931, 0.491549) (0.134782, 0.489057) (0.127037, 0.482759) (0.122579, 0.47898) (0.12069, 0.477332) (0.116233, 0.473844) (0.110265, 0.469125) (0.103448, 0.463604) (0.0977963, 0.45958) (0.0884803, 0.452823) (0.0862069, 0.451204) (0.0851085, 0.450473) (0.0816361, 0.448276) (0.0721458, 0.441915) (0.0689655, 0.439895) (0.063068, 0.436481) (0.0588076, 0.434109) (0.0517241, 0.430142) (0.0449837, 0.427274) (0.0402153, 0.425258) (0.0344828, 0.423001) (0.0305063, 0.421746) (0.0197797, 0.41887) (0.0172414, 0.418205) (0.0151473, 0.417981) (0.00145151, 0.416696) (0., 0.416577)};
\addplot [color=cyan,line width=0.7pt,smooth] coordinates {(1., 0) (0.992078, -0.0186396) (0.991379, -0.0192703) (0.985634, -0.0344828) (0.984032, -0.0370299) (0.982759, -0.0390545) (0.977859, -0.0517241) (0.975925, -0.0552984) (0.970729, -0.0689655) (0.967755, -0.0734406) (0.965517, -0.0769574) (0.95952, -0.0914539) (0.954618, -0.103448) (0.95122, -0.109336) (0.948276, -0.114497) (0.942853, -0.127085) (0.938152, -0.137931) (0.934418, -0.144697) (0.931034, -0.150805) (0.925913, -0.162171) (0.921226, -0.172414) (0.917338, -0.179504) (0.913793, -0.185889) (0.908691, -0.196693) (0.903768, -0.206897) (0.899971, -0.213734) (0.896552, -0.219778) (0.891175, -0.230625) (0.885704, -0.241379) (0.882302, -0.247362) (0.87931, -0.252502) (0.87335, -0.26394) (0.866956, -0.275862) (0.864317, -0.280357) (0.862069, -0.284087) (0.8552, -0.296608) (0.847427, -0.310345) (0.845999, -0.312688) (0.844828, -0.314556) (0.83671, -0.328592) (0.830557, -0.338886) (0.827586, -0.343841) (0.827326, -0.344307) (0.827025, -0.344828) (0.817839, -0.359816) (0.810345, -0.37162) (0.808254, -0.375128) (0.805661, -0.37931) (0.798567, -0.390238) (0.793103, -0.398351) (0.788776, -0.405139) (0.783035, -0.413793) (0.778878, -0.419825) (0.775862, -0.424213) (0.771178, -0.431034) (0.768868, -0.434288) (0.758922, -0.448276) (0.758743, -0.448521) (0.758621, -0.448687) (0.757981, -0.449555) (0.748468, -0.462453) (0.741379, -0.471655) (0.738064, -0.476127) (0.732914, -0.482759) (0.727529, -0.48954) (0.724138, -0.493622) (0.716858, -0.502682) (0.710263, -0.510509) (0.706897, -0.514454) (0.706028, -0.515504) (0.704513, -0.517241) (0.695017, -0.527966) (0.689655, -0.533726) (0.683853, -0.540119) (0.672747, -0.551724) (0.672527, -0.551951) (0.672414, -0.552062) (0.672076, -0.5524) (0.660967, -0.563314) (0.655172, -0.568694) (0.649227, -0.574317) (0.639574, -0.58292) (0.637931, -0.584372) (0.637288, -0.58492) (0.635736, -0.586207) (0.625082, -0.594991) (0.62069, -0.598515) (0.614279, -0.603448) (0.612674, -0.604658) (0.610767, -0.606051) (0.603448, -0.611226) (0.599961, -0.613716) (0.589437, -0.62069) (0.587009, -0.622294) (0.586207, -0.622786) (0.584692, -0.623719) (0.573716, -0.630191) (0.568966, -0.632764) (0.560677, -0.637266) (0.560149, -0.637539) (0.551724, -0.641477) (0.546175, -0.644074) (0.538392, -0.647354) (0.534483, -0.648883) (0.531852, -0.64991) (0.517376, -0.654904) (0.517241, -0.654946) (0.517143, -0.654976) (0.51641, -0.655172) (0.501936, -0.659044) (0.5, -0.659481) (0.497576, -0.66002) (0.486251, -0.662158) (0.482759, -0.662743) (0.478714, -0.663262) (0.470039, -0.664216) (0.465517, -0.664536) (0.460672, -0.664864) (0.453238, -0.665096) (0.448276, -0.665154) (0.443352, -0.665021) (0.435783, -0.664669) (0.431034, -0.664279) (0.426665, -0.66391) (0.417606, -0.662798) (0.413793, -0.662263) (0.410536, -0.661686) (0.398636, -0.659342) (0.396552, -0.658897) (0.394896, -0.658485) (0.387931, -0.656751) (0.382369, -0.655172) (0.379697, -0.6544) (0.37931, -0.65428) (0.378783, -0.654118) (0.364935, -0.64944) (0.362069, -0.648351) (0.357871, -0.646777) (0.350509, -0.64381) (0.344828, -0.64128) (0.336382, -0.637581) (0.335922, -0.63736) (0.327586, -0.633142) (0.322598, -0.630666) (0.31266, -0.625321) (0.310345, -0.62402) (0.309039, -0.623302) (0.30461, -0.62069) (0.29574, -0.615417) (0.293103, -0.613735) (0.288064, -0.610611) (0.282641, -0.607132) (0.277026, -0.603448) (0.275862, -0.602668) (0.269744, -0.598444) (0.262027, -0.593019) (0.258621, -0.590604) (0.257008, -0.589433) (0.252563, -0.586207) (0.25, -0.584305) (0.244436, -0.580093) (0.241379, -0.577683) (0.234528, -0.572505) (0.231976, -0.570531) (0.224138, -0.564163) (0.219656, -0.560688) (0.208616, -0.551724) (0.206897, -0.550279) (0.205669, -0.549269) (0.195248, -0.540539) (0.189655, -0.535649) (0.183134, -0.530285) (0.175798, -0.52401) (0.172414, -0.521119) (0.171052, -0.519966) (0.167807, -0.517241) (0.158991, -0.509603) (0.155172, -0.506172) (0.146897, -0.49931) (0.14569, -0.498277) (0.137931, -0.491549) (0.134782, -0.489057) (0.127037, -0.482759) (0.122579, -0.47898) (0.12069, -0.477332) (0.116233, -0.473844) (0.110265, -0.469125) (0.105566, -0.465517) (0.103448, -0.463841) (0.0977963, -0.45958) (0.0884803, -0.452823) (0.0862069, -0.451204) (0.0851085, -0.450473) (0.0818092, -0.448276) (0.0775862, -0.445372) (0.0721458, -0.441915) (0.0689655, -0.439753) (0.063068, -0.436481) (0.0588076, -0.434109) (0.0517241, -0.430142) (0.0449837, -0.427274) (0.0402153, -0.425258) (0.0344828, -0.422845) (0.0305063, -0.421746) (0.0197797, -0.41887) (0.0172414, -0.418205) (0.0151473, -0.417981) (0.00145151, -0.416696) (0., -0.416625)};
\draw [color=gray, dashed, line width=0.3pt] (axis cs: 0.45,0)--(axis cs: 0.45, 1);
\end{axis}
\end{tikzpicture}
\caption{Representation of the intersection in the setting of Theorem~\ref{th-main}.} 
\end{subfigure}
\hspace{5pt}
\begin{subfigure}[t]{0.3\textwidth}
\centering
\begin{tikzpicture}[scale=1]
\begin{axis}[
  scaled ticks=false,
  tick label style={font=\scriptsize},
  axis y line=left, axis x line=middle,
  xtick={0.45, 1},
  ytick={0},
  xticklabels={$\alpha$, $1$},
  yticklabels={$0$},
  xlabel={\small $v$},
  ylabel={\small $\phi(v')$},
every axis x label/.style={  at={(ticklabel* cs:1.0)},  anchor=west,
},
every axis y label/.style={  at={(ticklabel* cs:1.0)},  anchor=south,
},
  width=5cm,
  height=5cm,
  xmin=0,
  xmax=1.1,
  ymin=-1,
  ymax=1]
\addplot [color=white, fill=magenta, fill opacity=0.1, line width=0pt, smooth] coordinates {(0., 0.) (0.1,0.2) (0.2, 0.6) (0.3, 0.6) (0.45,0.95) (0.45,0.75) (0.3,0.45) (0.2, 0.35) (0.1,0.01) (0.,0.)};
\addplot [color=white, fill=white, line width=0pt] coordinates {(0.45, 0.01) (0.6, 0.01) (0.6, 1.) (0.45, 1.) (0.45,0.01)};
\addplot [color=magenta,line width=0.7pt,smooth] coordinates {(0., 0.) (0.1,0.05) (0.2,0.5) (0.3,0.5) (0.45,0.9)};
\addplot [color=cyan,line width=0.7pt,smooth] coordinates {(0., 2.63678*10^-16) (0.00840148, -0.0176798) (0.00862069, -0.0178412) (0.0163447, -0.0344828) (0.0168892, -0.035187) (0.0172414, -0.0356425) (0.0249135, -0.0517241) (0.0254558, -0.0525367) (0.0334928, -0.0689655) (0.0341027, -0.0697256) (0.0344828, -0.0701934) (0.0428318, -0.0867502) (0.0515323, -0.103448) (0.0516447, -0.103607) (0.0517241, -0.103716) (0.0539513, -0.107903) (0.0605473, -0.120285) (0.0689655, -0.135469) (0.0700735, -0.137931) (0.0786278, -0.153089) (0.0862069, -0.165946) (0.0878086, -0.16921) (0.0894629, -0.172414) (0.0970864, -0.185138) (0.103448, -0.195277) (0.106464, -0.200866) (0.109891, -0.206897) (0.115943, -0.21639) (0.12069, -0.223488) (0.125527, -0.231704) (0.131534, -0.241379) (0.13522, -0.246802) (0.137931, -0.250599) (0.145023, -0.261678) (0.154606, -0.275862) (0.154942, -0.276324) (0.155172, -0.276626) (0.156533, -0.278584) (0.165011, -0.290669) (0.172414, -0.300643) (0.175211, -0.30475) (0.179277, -0.310345) (0.181034, -0.312644) (0.185542, -0.318571) (0.189655, -0.323979) (0.192394, -0.327586) (0.196007, -0.332124) (0.206134, -0.344828) (0.20661, -0.3454) (0.206897, -0.345724) (0.208032, -0.347098) (0.217413, -0.358278) (0.224138, -0.365793) (0.228376, -0.370834) (0.236004, -0.37931) (0.239499, -0.383072) (0.241379, -0.385074) (0.247426, -0.391404) (0.250811, -0.39493) (0.258621, -0.402499) (0.262359, -0.406316) (0.270258, -0.413793) (0.27409, -0.417338) (0.275862, -0.418859) (0.280463, -0.422995) (0.286062, -0.427876) (0.293103, -0.433529) (0.298297, -0.437888) (0.309438, -0.446462) (0.310345, -0.447132) (0.310758, -0.44745) (0.311897, -0.448276) (0.318966, -0.45324) (0.323569, -0.456311) (0.327586, -0.458991) (0.335408, -0.463919) (0.336629, -0.464673) (0.338057, -0.465517) (0.344828, -0.469402) (0.350096, -0.472221) (0.359093, -0.476806) (0.362069, -0.47821) (0.363895, -0.479106) (0.372448, -0.482759) (0.378111, -0.485158) (0.37931, -0.485596) (0.381061, -0.486259) (0.392816, -0.490229) (0.396552, -0.491276) (0.401537, -0.492729) (0.408007, -0.494331) (0.413793, -0.495439) (0.420842, -0.496857) (0.42376, -0.497307) (0.431034, -0.498033) (0.439108, -0.498906) (0.440171, -0.498968) (0.448276, -0.499014) (0.456458, -0.499122) (0.457353, -0.499087) (0.465517, -0.498353) (0.473004, -0.497733) (0.47544, -0.497397) (0.482759, -0.496033) (0.488849, -0.494939) (0.494591, -0.493576) (0.5, -0.492046) (0.50408, -0.490919) (0.514996, -0.487249) (0.517241, -0.486397) (0.518776, -0.485827) (0.525845, -0.482759) (0.532937, -0.479666) (0.534483, -0.478888) (0.53705, -0.477623) (0.54661, -0.472531) (0.551724, -0.46965) (0.558488, -0.465517) (0.5599, -0.464628) (0.56123, -0.463748) (0.568966, -0.458346) (0.572717, -0.455779) (0.58262, -0.448276) (0.585215, -0.446293) (0.586207, -0.445455) (0.588607, -0.443476) (0.597315, -0.43601) (0.603448, -0.430253) (0.609092, -0.425081) (0.620382, -0.413793) (0.620586, -0.413586) (0.62069, -0.413472) (0.621037, -0.413099) (0.631706, -0.401343) (0.637931, -0.393911) (0.64255, -0.388549) (0.649915, -0.37931) (0.653122, -0.37521) (0.655172, -0.372375) (0.663421, -0.361324) (0.666108, -0.35744) (0.672414, -0.348216) (0.673406, -0.346812) (0.674712, -0.344828) (0.683106, -0.33173) (0.689655, -0.320742) (0.69253, -0.316094) (0.695844, -0.310345) (0.701668, -0.299888) (0.706897, -0.289794) (0.710512, -0.283093) (0.714149, -0.275862) (0.719049, -0.265683) (0.724138, -0.254303) (0.727262, -0.247627) (0.729983, -0.241379) (0.735132, -0.228884) (0.741379, -0.212534) (0.742635, -0.209407) (0.743572, -0.206897) (0.748558, -0.192538) (0.74973, -0.189115) (0.754997, -0.172414) (0.756324, -0.16782) (0.758621, -0.159123) (0.762417, -0.145523) (0.764358, -0.137931) (0.767948, -0.122102) (0.768931, -0.11731) (0.771652, -0.103448) (0.772707, -0.0971374) (0.775862, -0.0756874) (0.776663, -0.070567) (0.776883, -0.0689655) (0.777157, -0.0663754) (0.779477, -0.0417129) (0.780034, -0.0344828) (0.780446, -0.0253158) (0.780923, -0.0101211) (0.781085, 0) (0.780923, 0.0101211) (0.780446, 0.0253158) (0.780034, 0.0344828) (0.779477, 0.0417129) (0.777157, 0.0663754) (0.776883, 0.0689655) (0.776663, 0.070567) (0.775862, 0.0756874) (0.772707, 0.0971374) (0.771652, 0.103448) (0.768931, 0.11731) (0.767948, 0.122102) (0.767241, 0.125223) (0.764362, 0.137931) (0.762417, 0.145523) (0.758621, 0.159123) (0.756324, 0.16782) (0.754997, 0.172414) (0.74973, 0.189115) (0.748558, 0.192538) (0.743572, 0.206897) (0.742635, 0.209407) (0.741379, 0.212534) (0.735132, 0.228884) (0.729983, 0.241379) (0.727262, 0.247627) (0.724138, 0.254303) (0.719049, 0.265683) (0.714149, 0.275862) (0.710512, 0.283093) (0.706897, 0.289794) (0.701668, 0.299888) (0.695844, 0.310345) (0.69253, 0.316094) (0.689655, 0.320742) (0.683106, 0.33173) (0.674712, 0.344828) (0.673406, 0.346812) (0.672414, 0.348216) (0.666108, 0.35744) (0.663421, 0.361324) (0.662893, 0.362069) (0.655172, 0.372518) (0.653122, 0.37521) (0.649915, 0.37931) (0.64255, 0.388549) (0.637931, 0.393911) (0.631706, 0.401343) (0.621037, 0.413099) (0.62069, 0.413472) (0.620586, 0.413586) (0.620382, 0.413793) (0.609092, 0.425081) (0.603448, 0.430253) (0.597315, 0.43601) (0.588607, 0.443476) (0.586207, 0.445502) (0.585215, 0.446293) (0.58262, 0.448276) (0.572717, 0.455779) (0.568966, 0.458346) (0.56123, 0.463748) (0.5599, 0.464628) (0.551724, 0.46945) (0.54661, 0.472531) (0.53705, 0.477623) (0.534483, 0.478888) (0.532937, 0.479666) (0.525845, 0.482759) (0.517241, 0.486397) (0.514996, 0.487249) (0.50408, 0.490919) (0.5, 0.492046) (0.494591, 0.493576) (0.488849, 0.494939) (0.482759, 0.496218) (0.47544, 0.497397) (0.473004, 0.497733) (0.465517, 0.498353) (0.457353, 0.499087) (0.456458, 0.499122) (0.448276, 0.499241) (0.440171, 0.498968) (0.439108, 0.498906) (0.431034, 0.498033) (0.42376, 0.497307) (0.420842, 0.496857) (0.413793, 0.495617) (0.408007, 0.494331) (0.401537, 0.492729) (0.396552, 0.491276) (0.392816, 0.490229) (0.381061, 0.486259) (0.37931, 0.485596) (0.378111, 0.485158) (0.372448, 0.482759) (0.363895, 0.479106) (0.362069, 0.47821) (0.359093, 0.476806) (0.350096, 0.472221) (0.344828, 0.469199) (0.336629, 0.464673) (0.335408, 0.463919) (0.327586, 0.458991) (0.323569, 0.456311) (0.311936, 0.448276) (0.310758, 0.44745) (0.310345, 0.447151) (0.309438, 0.446462) (0.298297, 0.437888) (0.293103, 0.433529) (0.286062, 0.427876) (0.280463, 0.422995) (0.275862, 0.418859) (0.27409, 0.417338) (0.270258, 0.413793) (0.262359, 0.406316) (0.258621, 0.402499) (0.250811, 0.39493) (0.247426, 0.391404) (0.241379, 0.384966) (0.239499, 0.383072) (0.236004, 0.37931) (0.228376, 0.370834) (0.224138, 0.365793) (0.217413, 0.358278) (0.208032, 0.347098) (0.206897, 0.345724) (0.20661, 0.3454) (0.206134, 0.344828) (0.196007, 0.332124) (0.192394, 0.327586) (0.189655, 0.323979) (0.185542, 0.318571) (0.17926, 0.310345) (0.175211, 0.30475) (0.172414, 0.300899) (0.166701, 0.293103) (0.165011, 0.290669) (0.156533, 0.278584) (0.155172, 0.276626) (0.154942, 0.276324) (0.154606, 0.275862) (0.145023, 0.261678) (0.137931, 0.250599) (0.13522, 0.246802) (0.131534, 0.241379) (0.125527, 0.231704) (0.12069, 0.223488) (0.115943, 0.21639) (0.109891, 0.206897) (0.106464, 0.200866) (0.103448, 0.195277) (0.0970864, 0.185138) (0.0894629, 0.172414) (0.0878086, 0.16921) (0.0862069, 0.165946) (0.0786278, 0.153089) (0.0700735, 0.137931) (0.0695414, 0.136779) (0.0689655, 0.135469) (0.0605473, 0.120285) (0.0539513, 0.107903) (0.0517241, 0.103716) (0.0516447, 0.103607) (0.0515323, 0.103448) (0.0428318, 0.0867502) (0.0344828, 0.0701934) (0.0341027, 0.0697256) (0.0334928, 0.0689655) (0.0254558, 0.0525367) (0.0249135, 0.0517241) (0.0172414, 0.0356425) (0.0168892, 0.035187) (0.0159049, 0.0344828) (0.00840148, 0.0176798) (0., 0.)};
\draw [color=gray, dashed, line width=0.3pt] (axis cs: 0.45,0)--(axis cs: 0.45, 1);
\end{axis}
\end{tikzpicture}
\caption{Representation of the intersection in the setting of Theorem~\ref{th-main-h}.} 
\end{subfigure}
\hspace{5pt}
\begin{subfigure}[t]{0.3\textwidth}
\centering
\begin{tikzpicture}[scale=1]
\begin{axis}[
  scaled ticks=false,
  tick label style={font=\scriptsize},
  axis y line=left, axis x line=middle,
  xtick={0.45, 1},
  ytick={0},
  xticklabels={$\alpha$, $1$},
  yticklabels={$0$},
  xlabel={\small $v$},
  ylabel={\small $\phi(v')$},
every axis x label/.style={  at={(ticklabel* cs:1.0)},  anchor=west,
},
every axis y label/.style={  at={(ticklabel* cs:1.0)},  anchor=south,
},
  width=5cm,
  height=5cm,
  xmin=0,
  xmax=1.1,
  ymin=-1,
  ymax=1]
\addplot [color=white, fill=magenta, fill opacity=0.1, line width=0pt, smooth] coordinates {(0., 0.) (0.1,0.2) (0.2, 0.6) (0.3, 0.6) (0.45,0.95) (0.45,0.75) (0.3,0.45) (0.2, 0.35) (0.1,0.01) (0.,0.)};
\addplot [color=white, fill=white, line width=0pt] coordinates {(0.45, 0.01) (0.6, 0.01) (0.6, 1.) (0.45, 1.) (0.45,0.01)};
\addplot [color=magenta,line width=0.7pt,smooth] coordinates {(0., 0.) (0.1,0.05) (0.2,0.5) (0.3,0.5) (0.45,0.9)};
\addplot [color=cyan,line width=0.7pt,smooth] coordinates {(0.327586, -0.454427) (0.32528, -0.452888) (0.318966, -0.448676) (0.318395, -0.448276) (0.312487, -0.443992) (0.310345, -0.44234) (0.305639, -0.438863) (0.299985, -0.434512) (0.293103, -0.428737) (0.287772, -0.424457) (0.276111, -0.41429) (0.275862, -0.414067) (0.275766, -0.413985) (0.275559, -0.413793) (0.264062, -0.40291) (0.258621, -0.397354) (0.252544, -0.391464) (0.241925, -0.380401) (0.241379, -0.37983) (0.24121, -0.379649) (0.240895, -0.37931) (0.230123, -0.36734) (0.224138, -0.360222) (0.2192, -0.354704) (0.210913, -0.344828) (0.206897, -0.339715) (0.200202, -0.331438) (0.197835, -0.328469) (0.189655, -0.317562) (0.187425, -0.314805) (0.184018, -0.310345) (0.177156, -0.30086) (0.172414, -0.2939) (0.167025, -0.286639) (0.159465, -0.275862) (0.157029, -0.272149) (0.155172, -0.269145) (0.147166, -0.257392) (0.141686, -0.248889) (0.137931, -0.243077) (0.137447, -0.242346) (0.13679, -0.241379) (0.127872, -0.227014) (0.12069, -0.214815) (0.118424, -0.211428) (0.115536, -0.206897) (0.109105, -0.195583) (0.103448, -0.185098) (0.0999189, -0.179473) (0.0956897, -0.172414) (0.0908718, -0.163084) (0.0862069, -0.153575) (0.0819744, -0.146396) (0.0771974, -0.137931) (0.073244, -0.129374) (0.0689655, -0.119638) (0.0647091, -0.111961) (0.0601743, -0.103448) (0.0564191, -0.0940584) (0.0517241, -0.0817379) (0.0484641, -0.0754856) (0.0451758, -0.0689655) (0.0410264, -0.0558783) (0.0346337, -0.0347847) (0.0345489, -0.0344828) (0.0345322, -0.0343839) (0.0344828, -0.0338948) (0.0289979, 0) (0.0344828, 0.0338948) (0.0345322, 0.0343839) (0.0345489, 0.0344828) (0.0346337, 0.0347847) (0.0410264, 0.0558783) (0.0451758, 0.0689655) (0.0484641, 0.0754856) (0.0517241, 0.0817379) (0.0564191, 0.0940584) (0.0601743, 0.103448) (0.0647091, 0.111961) (0.0689655, 0.119638) (0.073244, 0.129374) (0.0771974, 0.137931) (0.0819744, 0.146396) (0.0862069, 0.153575) (0.0908718, 0.163084) (0.0956897, 0.172414) (0.0999189, 0.179473) (0.103448, 0.185098) (0.109105, 0.195583) (0.115536, 0.206897) (0.118424, 0.211428) (0.12069, 0.214815) (0.127872, 0.227014) (0.13679, 0.241379) (0.137447, 0.242346) (0.137931, 0.243024) (0.141686, 0.248889) (0.147166, 0.257392) (0.155172, 0.269145) (0.157029, 0.272149) (0.159465, 0.275862) (0.167025, 0.286639) (0.172414, 0.2939) (0.177156, 0.30086) (0.184018, 0.310345) (0.187425, 0.314805) (0.189655, 0.317562) (0.197835, 0.328469) (0.200202, 0.331438) (0.206897, 0.339715) (0.208434, 0.341753) (0.210913, 0.344828) (0.2192, 0.354704) (0.224138, 0.360222) (0.230123, 0.36734) (0.240895, 0.37931) (0.24121, 0.379649) (0.241379, 0.37982) (0.241925, 0.380401) (0.252544, 0.391464) (0.258621, 0.397354) (0.264062, 0.40291) (0.275559, 0.413793) (0.275766, 0.413985) (0.275862, 0.414067) (0.276111, 0.41429) (0.287772, 0.424457) (0.293103, 0.428737) (0.299985, 0.434512) (0.305639, 0.438863) (0.310345, 0.442438) (0.312487, 0.443992) (0.318602, 0.448276) (0.32528, 0.452888) (0.327586, 0.454427) (0.332076, 0.457256) (0.33839, 0.461151) (0.344828, 0.464705) (0.351843, 0.468728) (0.356152, 0.470924) (0.362069, 0.473716) (0.3657, 0.475496) (0.378389, 0.480916) (0.37931, 0.481272) (0.379929, 0.481521) (0.38357, 0.482759) (0.394676, 0.486509) (0.396552, 0.487035) (0.399055, 0.487764) (0.409942, 0.490461) (0.413793, 0.491316) (0.418485, 0.492142) (0.42578, 0.493268) (0.431034, 0.493791) (0.436866, 0.494422) (0.442286, 0.494739) (0.448276, 0.494941) (0.454323, 0.494853) (0.459573, 0.494647) (0.465517, 0.494112) (0.470968, 0.493661) (0.477778, 0.49272) (0.482759, 0.491918) (0.486903, 0.491048) (0.497061, 0.488636) (0.5, 0.487804) (0.502217, 0.487192) (0.515403, 0.482759) (0.516976, 0.482227) (0.517241, 0.48212) (0.517639, 0.481964) (0.531142, 0.476076) (0.534483, 0.474393) (0.540031, 0.471661) (0.544883, 0.469077) (0.551724, 0.464956) (0.558165, 0.461157) (0.564682, 0.456843) (0.568966, 0.453851) (0.571043, 0.45243) (0.576526, 0.448276) (0.583531, 0.442923) (0.586207, 0.44079) (0.592684, 0.435322) (0.595685, 0.432749) (0.603448, 0.425462) (0.607449, 0.421795) (0.615453, 0.413793) (0.618929, 0.410272) (0.62069, 0.408327) (0.6266, 0.401972) (0.630095, 0.398121) (0.637931, 0.388765) (0.640922, 0.385293) (0.645692, 0.37931) (0.651475, 0.371916) (0.655172, 0.366804) (0.661753, 0.357988) (0.670856, 0.344828) (0.671751, 0.343501) (0.672414, 0.34244) (0.677671, 0.334313) (0.68145, 0.328417) (0.689655, 0.314649) (0.690845, 0.312725) (0.692241, 0.310345) (0.698276, 0.299521) (0.699953, 0.296457) (0.701788, 0.293103) (0.706897, 0.283251) (0.708762, 0.279593) (0.710638, 0.275862) (0.71726, 0.262106) (0.724138, 0.246728) (0.725431, 0.243965) (0.726569, 0.241379) (0.732759, 0.226342) (0.733253, 0.225128) (0.736461, 0.216734) (0.74018, 0.206897) (0.740677, 0.205492) (0.741379, 0.203334) (0.747662, 0.184979) (0.751624, 0.172414) (0.754186, 0.163545) (0.758621, 0.146753) (0.760201, 0.141092) (0.761009, 0.137931) (0.762991, 0.12919) (0.765574, 0.117354) (0.767241, 0.10888) (0.768309, 0.103448) (0.770196, 0.0921169) (0.77266, 0.0753696) (0.773533, 0.0689655) (0.773942, 0.0651258) (0.775862, 0.0434216) (0.77657, 0.0358982) (0.776679, 0.0344828) (0.776759, 0.0326881) (0.777672, 0.00361998) (0.77773, 0) (0.777672, -0.00361998) (0.776759, -0.0326881) (0.776679, -0.0344828) (0.77657, -0.0358982) (0.775862, -0.0434216) (0.773942, -0.0651258) (0.773533, -0.0689655) (0.77266, -0.0753696) (0.770196, -0.0921169) (0.768309, -0.103448) (0.767241, -0.10888) (0.765574, -0.117354) (0.762991, -0.12919) (0.761009, -0.137931) (0.760201, -0.141092) (0.758621, -0.146753) (0.754186, -0.163545) (0.751624, -0.172414) (0.747662, -0.184979) (0.741379, -0.203334) (0.740677, -0.205492) (0.74018, -0.206897) (0.736461, -0.216734) (0.733253, -0.225128) (0.732759, -0.226342) (0.726569, -0.241379) (0.725431, -0.243965) (0.724138, -0.2469) (0.718914, -0.258621) (0.71726, -0.262106) (0.715517, -0.26578) (0.710665, -0.275862) (0.708762, -0.279593) (0.706897, -0.283251) (0.701788, -0.293103) (0.699953, -0.296457) (0.692217, -0.310345) (0.690845, -0.312725) (0.689655, -0.314754) (0.681971, -0.327586) (0.68145, -0.328417) (0.677671, -0.334313) (0.672414, -0.34244) (0.671751, -0.343501) (0.670856, -0.344828) (0.661753, -0.357988) (0.655172, -0.366804) (0.651475, -0.371916) (0.645692, -0.37931) (0.640922, -0.385293) (0.637931, -0.388765) (0.630095, -0.398121) (0.6266, -0.401972) (0.62069, -0.408327) (0.618929, -0.410272) (0.615453, -0.413793) (0.607449, -0.421795) (0.603448, -0.425462) (0.595685, -0.432749) (0.592684, -0.435322) (0.586207, -0.440663) (0.583531, -0.442923) (0.576526, -0.448276) (0.571043, -0.45243) (0.568966, -0.453851) (0.564682, -0.456843) (0.558165, -0.461157) (0.551724, -0.465214) (0.551202, -0.465517) (0.544883, -0.469077) (0.540031, -0.471661) (0.534483, -0.474393) (0.531142, -0.476076) (0.517639, -0.481964) (0.517241, -0.48212) (0.516976, -0.482227) (0.515403, -0.482759) (0.502217, -0.487192) (0.5, -0.487804) (0.497061, -0.488636) (0.486903, -0.491048) (0.482759, -0.491792) (0.477778, -0.49272) (0.470968, -0.493661) (0.465517, -0.494112) (0.459573, -0.494647) (0.454323, -0.494853) (0.448276, -0.494773) (0.442286, -0.494739) (0.436866, -0.494422) (0.431034, -0.493791) (0.42578, -0.493268) (0.418485, -0.492142) (0.413793, -0.491198) (0.409942, -0.490461) (0.399055, -0.487764) (0.396552, -0.487035) (0.394676, -0.486509) (0.38357, -0.482759) (0.37931, -0.481272) (0.378389, -0.480916) (0.3657, -0.475496) (0.362069, -0.473716) (0.356152, -0.470924) (0.351843, -0.468728) (0.34584, -0.465517) (0.344828, -0.464959) (0.33839, -0.461151) (0.332076, -0.457256) (0.327586, -0.454427)};
\draw [color=gray, dashed, line width=0.3pt] (axis cs: 0.45,0)--(axis cs: 0.45, 1);
\end{axis}
\end{tikzpicture}
\caption{Representation of the intersection in the setting of Theorem~\ref{th-main-h-2}.} 
\end{subfigure}
\captionof{figure}{Qualitative representation of the strategy of the proofs of the main existence results about heteroclinic, homoclinic, and ``definitively periodic'' solutions.}
\label{fig-06}
\end{figure}
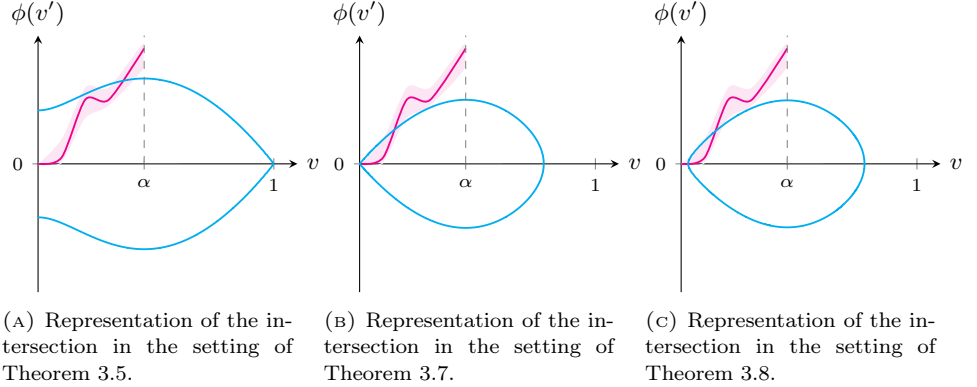

\begin{remark}[The case $F(1)=0$]\label{rem-F(1)=0}
We provide some comments regarding the balanced case $F(1)=0$, which was extensively dealt with by means of variational methods in \cite{BoCoNy-17}. Here the entire discussion can be carried out in the same way as above, but the sufficient condition \eqref{cond-c} for existence becomes $c \leq \eta$, necessarily implying that for every $t > t_{0}$ and $s \leq t_{0}$, it holds $q(t) \leq q(s)$. Under this assumption, we are able to find a heteroclinic solution. Comparing with \cite[Corollary 1.4]{BoCoNy-17}, we actually see that the two results overlap only for some precise choices of the weight $q$, but are in general quite different. Due to the technique used, in our result we allow $q$ to have a general behavior on the left of $t_{0}$ but we have to require $q$ constant on the right, while \cite[Corollary 1.4]{BoCoNy-17} exploits the assumption that $q$ asymptotically converges to its upper bound both at $\pm\infty$, leaving more freedom in between. However, as already remarked, our result holds in the case $F(1) > 0$ as well, differently from \cite[Corollary 1.4]{BoCoNy-17}. 
\hfill$\lhd$
\end{remark}

\subsection{Behavior of heteroclinics and homoclinics in dependence on $\delta$}\label{section-3.4}

In this section, we analyze the behavior of heteroclinic and homoclinic solutions of \eqref{eq-sect-3} in dependence on $\delta$, taking into account both the cases $\delta \to 0^{+}$ (\emph{vanishing diffusion}) and $\delta \to +\infty$ (\emph{large diffusion}). First, we focus on heteroclinics. 

\begin{theorem}\label{th-main-2}
For any $\delta > 0$, denote by $v_\delta$ the increasing heteroclinic solution of \eqref{eq-sect-3} provided by Theorem~\ref{th-main} and let $v_{*}=\lim_{\delta \to 0^{+}} v_\delta(t_{0}) \in \mathopen{]}0,\alpha\mathclose{]}$. Then, for every $t \in \mathbb{R}$ it holds that
\begin{equation*} 
\lim_{\delta \to 0^{+}} v_\delta(t) =\hat{v}(t):=
\begin{cases}
\, 0, &\text{if $t\in\mathopen{]}-\infty,t_{0}-v^*\mathclose{[}$,}
\\
\, t-t_{0}+v_{*}, & \text{if $t\in\mathopen{[}t_{0}-v_{*},t_{0}-v_{*}+1\mathclose{]}$,}
\\
\, 1, &\text{if $t\in\mathopen{]}t_{0}-v_{*}+1,+\infty\mathclose{]}$,}
\end{cases}
\end{equation*}
\begin{equation*}
\lim_{\delta \to +\infty} v_\delta(t)=v_{*},
\end{equation*}
where the former convergence is uniform on $\mathbb{R}$, while the second is locally uniform.
\end{theorem}
\begin{proof}
We preliminarily observe that by the Ascoli--Arzel\`a theorem there exist nondecreasing Lipschitz continuous functions $\hat{v}_{0}$ and $\hat{v}_\infty$ such that $v_{\delta} \to \hat{v}_{0}$ for $\delta \to 0^{+}$ and $v_{\delta} \to \hat{v}_\infty$ for $\delta \to +\infty$, where both the convergences are locally uniform in $\mathbb{R}$.

As for the case $\delta \to 0^{+}$, we first work in the time interval $\mathopen{]}-\infty,t_{0}\mathclose{[}$; we recall that $v_\delta < \alpha$ on such an interval, due to the construction in the previous sections. Since $v_\delta$ is strictly increasing on $\mathopen{]}-\infty,t_{0}\mathclose{[}$ for every $\delta$, the corresponding function $y_\delta$ given by \eqref{eq-y} is well defined and 
satisfies the problem 
\begin{equation*}
\begin{cases}
\, \dot{y}_\delta(v)=-\dfrac{q(t_\delta(v))f(v)}{\delta},
\\
\, y_\delta(0)=0,
\end{cases}
\end{equation*}
for $v \in \mathopen{[}0,v_\delta(t_{0})\mathclose{]}$. 
On any compact set $\mathopen{[}v_1,v_2\mathclose{]} \subset \mathopen{]}0,v_{*}\mathclose{[}$, one has $\dot{y}_\delta(v) \geq -\frac{\eta}{\delta} f(v)$ and hence $y_\delta(v) \geq -\frac{\eta}{\delta} F(v)$, implying $y_\delta \to +\infty$ uniformly, for $\delta \to 0^{+}$.
Consequently, fixed $\mathopen{[}v_1,v_2\mathclose{]} \subset \mathopen{]}0,v_{*}\mathclose{[}$ and letting $\mathopen{[}t_1,t_2\mathclose{]}=\hat{v}_{0}^{-1}(\mathopen{[}v_1,v_2\mathclose{]})$, one has
\begin{equation}\label{eq-formuladeriv}
v_\delta'(t)=\frac{\sqrt{y_\delta(v_\delta(t))(2+y_\delta(v_\delta(t)))}}{1+y_\delta(v_\delta(t))} \to 1, \quad \text{as $\delta\to0^{+}$,}
\end{equation}
for every $t \in \mathopen{[}t_1, t_2\mathclose{]}$. However, $\{v_{\delta}'\}_\delta$ is bounded in $L^2_{\text{loc}}(\mathbb{R})$, so (up to subsequences) it has a weak limit $w \in L^2_{\mathrm{loc}}(\mathbb{R})$ satisfying $0 \leq w \leq 1$, which coincides with the distributional derivative of $\hat{v}_{0}$. 

Proceeding as in Section~\ref{section-2}, thanks to the dominated convergence theorem, we then have
\begin{equation*}
\int_{t_1}^{t_2} \mathrm{d}s \geq \int_{t_1}^{t_2} w(s) \, \mathrm{d}s 
= \hat{v}_{0}(t_2)-\hat{v}_{0}(t_1) 
= \lim_{\delta \to 0^{+}}\int_{t_1}^{t_2} v_{\delta}'(s) \, \mathrm{d}s 
= \int_{t_1}^{t_2} \mathrm{d}s 
\end{equation*}
and hence $w(t)=1$ for almost every $t \in \mathopen{[}t_1, t_2\mathclose{]}$. Repeating the argument for every $v_1, v_2$, one has that the distributional derivative of $\hat{v}_{0}$ coincides almost everywhere with $1$ whenever $\hat{v}_{0}$ is strictly positive on $\mathopen{]}-\infty,t_{0}\mathclose{[}$. Being $\hat{v}_{0}$ absolutely continuous, for every $t\in\mathopen{]}-\infty,t_{0}\mathclose{[}$ such that $\hat{v}(t) > 0$ we have that
\begin{equation*}
v_{*}-\hat{v}_{0}(t)=\hat{v}_{0}(t_{0})-\hat{v}_{0}(t)=\int_t^{t_{0}} w(s) \,\mathrm{d}s=t_{0}-t,
\end{equation*}
whence the conclusion follows (recall that $\hat{v}_{0}$ is nondecreasing). In particular, $\hat{v}_{0}$ has to be identically equal to $0$ on the left of the time $t$ in which $t-t_{0}+v_{*}$ vanishes.

On the other hand, on $\mathopen{]}v_\delta(t_{0}), 1\mathclose{[}$ the function $y_\delta$ given by \eqref{eq-y} satisfies
\begin{equation*}
\begin{cases}
\, \dot{y}_\delta(v)=-\dfrac{cf(v)}{\delta},
\\
\, y_\delta(1)=0,
\end{cases}
\end{equation*}
and hence it is explicitly given by $y_\delta(v)=-\frac{c}{\delta}(F(v)-F(1))$. 
Fixed $\mathopen{[}w_1,w_2\mathclose{]} \subset \mathopen{]}v_{*}, 1\mathclose{[}$ and letting $\mathopen{[}\tau_1,\tau_2\mathclose{]}=\hat{v}_{0}^{-1}(\mathopen{[}w_1,w_2\mathclose{]})$, one then has
\begin{align*}
v_\delta'(t)
&=\frac{\sqrt{y_\delta(v_\delta(t))(2+y_\delta(v_\delta(t)))}}{1+y_\delta(v_\delta(t))}
\\
&=
\frac{\sqrt{c(F(1)-F(v_\delta(t)))(2\delta+c(F(1)-F(v_\delta(t))))}}{\delta+c(F(1)-F(v_\delta(t)))}
\to 1, 
\quad \text{as $\delta\to0^{+}$,}
\end{align*}
for every $t \in \mathopen{[}\tau_1, \tau_2\mathclose{]}$.
The argument can then be concluded as before: for any time $t > t_{0}$ for which $\hat{v}_{0}(t) \in \mathopen{]}v_{*}, 1\mathclose{[}$, the distributional derivative of $\hat{v}_{0}$ is equal to $1$ and hence $\hat{v}_{0}$ coincides with the function in the statement. The uniform convergence follows as in \cite[Lemma~2.4]{Di-76}.
\smallbreak
As for the case $\delta \to +\infty$, here the function $y_\delta$ defined on $\mathopen{[}0,1\mathclose{]}$ trivially satisfies $y_\delta \to 0$ uniformly, implying via \eqref{eq-formuladeriv} that $v_\delta' \to 0$ locally uniformly. Consequently, $v_\delta$ converges locally uniformly to a constant, which is necessarily equal to $\hat{v}_\infty(t_{0})=v_{*}$; notice that the convergence is not uniform on the whole $\mathbb{R}$ since $v_\delta(-\infty)=0$, $v_\delta(+\infty)=1$ for every $\delta$. 
\end{proof}

Similarly, one can discuss the behavior of homoclinic solutions as $\delta \to 0^{+}$ and $\delta \to +\infty$. Here one can proceed as in the proof of Theorem~\ref{th-main-2}, first considering $t \in \mathopen{]}-\infty,t_{0}\mathclose{[}$ and then working with the autonomous problem in the complementary interval. This gives rise to the following statement, which is in accord with Proposition~\ref{th-delta-0}. 

\begin{theorem}\label{th-main-3}
For any $\delta > 0$, denote by $v_\delta$ the increasing heteroclinic solution of \eqref{eq-sect-3} provided by Theorem~\ref{th-main-h} and let $v_{*}=\lim_{\delta \to 0^{+}} v_\delta(t_{0}) \in \mathopen{]}0,\alpha\mathclose{]}$. Then, for every $t \in \mathbb{R}$ it holds that
\begin{equation*} 
\lim_{\delta \to 0^{+}} v_\delta(t) =\hat{v}(t):=
\begin{cases}
\, 0, &\text{if $t\in\mathopen{]}-\infty,t_{0}-v^*\mathclose{[} \cup \mathopen{]}v^*-t_{0}, +\infty\mathclose{[} $,}
\\
\, t-t_{0}+v_{*}, & \text{if $t\in\mathopen{[}t_{0}-v_{*},t_{0}-v_{*}+v_{0}\mathclose{]}$,}
\\
\, -t+t_{0}-v_{*}, &\text{if $t\in\mathopen{]}t_{0}-v_{*}+v_{0},t_{0}-v_{*}\mathclose{]}$,}
\end{cases}
\end{equation*}
\begin{equation*}
\lim_{\delta \to +\infty} v_\delta(t)=v_{*},
\end{equation*}
where the former convergence is uniform on $\mathbb{R}$, while the second is locally uniform.
\end{theorem}

\bibliographystyle{elsart-num-sort}
\bibliography{FeGa-biblio}
\nocite*{}

\end{document}